\documentclass[11pt]{amsart}
\usepackage[T1]{fontenc}
\usepackage{lmodern,microtype,amsmath,amssymb,amsthm,mathrsfs}
\usepackage[margin=1in]{geometry}
\usepackage{tikz,needspace}
\usetikzlibrary{arrows.meta}
\usepackage[hidelinks]{hyperref}
\allowdisplaybreaks
\numberwithin{equation}{section}
\newtheorem{theorem}{Theorem}[section]
\newtheorem{proposition}[theorem]{Proposition}
\newtheorem{lemma}[theorem]{Lemma}
\newtheorem{corollary}[theorem]{Corollary}

\theoremstyle{remark}
\newtheorem{example}[theorem]{Example}
\newcommand{\C}{\mathbb C}
\newcommand{\Q}{\mathbb Q}
\newcommand{\Z}{\mathbb Z}
\newcommand{\K}{\mathbb K}
\newcommand{\ct}{\operatorname{ct}}
\newcommand{\diag}{\operatorname{diag}}
\newcommand{\Inv}{\operatorname{Inv}}
\newcommand{\Tr}{\operatorname{Tr}}
\newcommand{\End}{\operatorname{End}}
\newcommand{\GL}{\mathrm{GL}}
\newcommand{\gl}{\mathfrak{gl}}
\newcommand{\slie}{\mathfrak{sl}}
\newcommand{\Sym}{\operatorname{Sym}}
\newcommand{\Tw}{\operatorname{Tw}}
\newcommand{\dd}{\mathrm d}
\newcommand{\qb}[2]{\begin{bmatrix}#1\\#2\end{bmatrix}_{q}}
\newcommand{\Gb}[2]{\begin{bmatrix}#1\\#2\end{bmatrix}^{\mathrm{Gauss}}_{q^2}}
\newcommand{\QB}[2]{\begin{bmatrix}#1\\#2\end{bmatrix}^{\mathrm{Gauss}}_{Q}}
\newcommand{\cP}{\mathcal P}
\newcommand{\cV}{\mathcal V}
\newcommand{\cK}{\mathcal K}
\newcommand{\vac}{\varnothing}
\newcommand{\R}{\mathscr R}
\newcommand{\br}[1]{\{#1\}_q}
\newcommand{\eps}{\varepsilon}
\DeclareMathOperator{\Spec}{Spec}
\hypersetup{pdftitle={Rectangular cyclotomic expansions of colored HOMFLY-PT polynomials},pdfauthor={Honghuai Fang and Tian Zhou},pdfsubject={The KNTZ double-twist formula and the general-knot cyclotomic expansion},pdfkeywords={colored HOMFLY-PT polynomial, rectangular representation, cyclotomic expansion, KNTZ, integral interpolation, KZ equation}}
\title[Rectangular cyclotomic expansions]{Rectangular cyclotomic expansions\\of colored HOMFLY--PT polynomials}
\author{Honghuai Fang}
\address{Institute for Theoretical Sciences, Westlake University, No.~600 Dunyu Road, Xihu District, Hangzhou, Zhejiang 310030, China}
\email{fanghonghuai@westlake.edu.cn}
\author{Tian Zhou}
\address{Peking University, No.~5 Yiheyuan Road, Haidian District, Beijing 100871, China}
\email{2201110034@pku.edu.cn}
\subjclass[2020]{Primary 57K14; Secondary 57K16, 17B37, 05E05, 33D80}
\keywords{colored HOMFLY--PT polynomial, rectangular representation, cyclotomic expansion, double-twist knot, Laurent integrality, Knizhnik--Zamolodchikov equation}
\date{}
\begin{document}
\begin{abstract}
We prove the rectangular cyclotomic expansion conjecture of Kameyama, Nawata, Tao, and Zhang for reduced colored HOMFLY--PT polynomials of arbitrary zero-framed knots. We also prove their explicit double-twist formula for every rectangular color and all pairs of nonzero integer twist parameters. For each knot, we construct a double Newton expansion with coefficients in $\mathbb Z[A^{\pm1},q^{\pm1}]$. These coefficients are uniquely determined and independent of the width and height of the rectangular color. The construction uses finite-rank representation symmetries and integral interpolation. The double-twist formula follows from a simultaneous exterior-power comparison of the two four-point Casimir operators and the known one-row formula. Finite transition formulas express the double Newton coefficients of double-twist knots in terms of the prescribed twist coefficients.
\end{abstract}
\maketitle

\section{Introduction}\label{sec:introduction}
For a knot $K$, the reduced colored HOMFLY--PT polynomials
$H_\lambda(K;A,q)$ form a family indexed by partitions $\lambda$.
For an admissible color, setting $A=q^N$ gives the corresponding
type-$A$ Reshetikhin--Turaev invariant, while the Hecke-skein construction keeps
$A$ and $q$ independent \cite{RT,AM}. For example, at $A=q^2$ the
one-row colors $(r)$ give the colored Jones polynomials associated with
the $(r+1)$-dimensional representations. We study how the invariant
depends on a rectangular color and seek formulas that separate the
knot from the width and height of the rectangle.

A cyclotomic expansion expresses a colored invariant as a finite sum
of explicit color-dependent kernels multiplied by coefficients
attached to the knot. Its arithmetic content lies in the integrality
of these coefficients. Factoring the kernels then gives divisibility
relations between different colors, and the expansion can be
specialized term by term. Habiro's integral expansion of the colored
Jones polynomial provides the model for this approach and is a basic
ingredient in his unified Witten--Reshetikhin--Turaev invariant
\cite{Habiro}. In the HOMFLY--PT setting, one seeks such an expansion
over $\Z[A^{\pm1},q^{\pm1}]$, so that the dependence on the quantum
group rank is retained in the variable $A$.
The ordinary invariant of the figure-eight knot already illustrates
this factorization:
\[
 H_{(1)}(4_1;A,q)
 =1+(Aq-A^{-1}q^{-1})(Aq^{-1}-A^{-1}q).
\]
The two factors show directly that its value is one at $A=q$ and
$A=q^{-1}$. Cyclotomic expansions
organize analogous factors across the entire sequence of colors.

For symmetric, or one-row, colors, Nawata and Oblomkov formulated a
cyclotomic expansion conjecture \cite[Conjecture~2.3]{NO}.
Kononov and Morozov studied the corresponding differential expansion
and the additional factors of its coefficients
\cite[equation~(2)]{KM}. The relation between these factors,
finite-rank representation symmetries, and the Alexander polynomial
is developed by Lanina and Morozov \cite[Sections~1--2]{LM}.
Fang and Zhou give a Laurent-integral Newton formulation of the
one-row expansion \cite[Theorem~8.4]{FZ}.
At fixed higher rank, Beliakova and Gorsky construct cyclotomic
expansions of quantum $\gl_N$ link invariants from central elements
and interpolation Macdonald polynomials \cite{BG}. Their work
organizes the dependence on highest weights at a fixed $N$; a
two-variable expansion instead organizes the rank specializations
within one Laurent-polynomial identity.

Write $(r^s)$ for the partition with $s$ rows of length $r$.
At rank $N>s$, it corresponds to the highest weight $r\omega_s$,
where $\omega_s$ is the $s$th fundamental weight.
The cases $s=1$ and $r=1$ give symmetric and exterior colors,
respectively, and transposition exchanges the two dimensions of a
rectangle. Kameyama, Nawata, Tao, and Zhang (KNTZ) proposed cyclotomic
expansions for all such colors \cite{KNTZ}. Their general-knot
conjecture asks for a sum over partitions contained in $(r^s)$,
with Laurent-polynomial coefficients independent of the outer width
$r$. The kernels are principal Schur specializations multiplied by
products over the boxes of the partition. For the first square color
$(2^2)=(2,2)$, the index set is
\[
 \varnothing,\quad (1),\quad (2),\quad (1,1),\quad (2,1),\quad (2,2).
\]
Thus the rectangular formula includes both row and column shapes,
even for this small color. For double-twist knots, KNTZ prescribe the
coefficients explicitly in terms of a triangular matrix of Gaussian
minors. The two twist parameters enter through separate factors in
each summand.

These are two different problems. For an arbitrary knot, one must
prove that the whole rectangular color sequence admits the proposed
integral expansion. For a double-twist knot, one must identify the
invariant with a specified finite sum. The latter formula includes
the trefoil and figure-eight knot and provides explicit evaluations
for every rectangle in an infinite two-parameter knot family.
Chen, Liu, and Zhu proved the double-twist formula for one-row colors
by skein methods \cite[Theorem~1.2]{CLZ}. Quantum $6j$-symbols give
the one-row exclusive Racah matrices \cite{MMS}, and rectangular
Racah matrices and triangular evolution matrices are studied in
\cite{Morozov18,Morozov20}. A rectangular theorem must extend the
one-row formulas to arbitrary height while retaining the integral
coefficient ring and the prescribed twist dependence.

We prove both rectangular statements of Kameyama--Nawata--Tao--Zhang.
For every knot, we obtain a unique Laurent expansion in the width at
fixed height and, more generally, a unique double expansion whose
coefficients are independent of both dimensions. We also describe
all partition-indexed Laurent expansions with the prescribed KNTZ
kernels. For double-twist knots, we prove the specified formula for
all rectangular colors and all nonzero integral twist parameters.
A finite integral change of kernels relates its explicit twist
coefficients to the coefficients of the general expansion. This
comparison connects closed formulas for a knot family with the
integral structure of rectangular invariants for arbitrary knots.

\subsection{Expansions for arbitrary knots}

Throughout, $K$ is an oriented knot in $S^3$ with framing zero,
and $U$ denotes the unknot. The ordinary HOMFLY--PT polynomial $P_L$
is normalized by
\[
 A P_{L_+}-A^{-1}P_{L_-}=(q-q^{-1})P_{L_0},\qquad P_U=1.
\]
For a partition $\lambda$ of $n$, let $\ell(\lambda)$ denote its
number of nonzero parts, and let $\mathcal H_\lambda(K;A,q)$ be the unreduced colored HOMFLY--PT
invariant: in the Hecke-skein construction, one takes the $n$-parallel
cable of $K$, inserts a primitive Hecke idempotent of type $\lambda$,
and evaluates with framing zero \cite{AM}. We define
\[
 H_\lambda(K;A,q)
 =\frac{\mathcal H_\lambda(K;A,q)}
        {\mathcal H_\lambda(U;A,q)}.
\]
Thus $H_\lambda(U;A,q)=1$, and $H_{(1)}(K;A,q)=P_K(A,q)$.
For $\ell(\lambda)\leq N$, the specialization $A=q^N$ is the
reduced Reshetikhin--Turaev invariant colored by the irreducible
polynomial $U_q(\gl_N)$-module of highest weight $\lambda$ \cite{RT}.
The empty color has value $H_\varnothing=1$.
For $r,s\geq0$, write
\[
 H_{r,s}(K;A,q)=H_{(r^s)}(K;A,q),\qquad
 H_{r,0}=H_{0,s}=1.
\]
We suppress $K,A,q$ when they are fixed.

Our coefficient ring and field are
\[
 \R=\Z[A^{\pm1},q^{\pm1}],\qquad \Q(A,q),
\]
where $A$ and $q$ are independent indeterminates. Put
\[
 \{x\}=x-x^{-1},\qquad
 \br{a}=\{q^a\},\qquad
 [a]_q=\frac{\br{a}}{\br{1}},\qquad
 D_a=\{Aq^a\}.
\]
For $a\geq0$, set $[a]_q!=\prod_{j=1}^a[j]_q$ and
$\br{a}!=\prod_{j=1}^a\br{j}$. For $0\leq b\leq a$, the symmetric
quantum binomial and the Gaussian binomial with parameter $Q$ are
\[
 \qb ab=\frac{[a]_q!}{[b]_q![a-b]_q!},\qquad
 \QB ab=\prod_{j=1}^b\frac{1-Q^{a-b+j}}{1-Q^j}.
\]
They are related by
\begin{equation}\label{eq:binomial}
 \Gb ab=q^{b(a-b)}\qb ab.
\end{equation}
For example, $\qb21=q+q^{-1}$, whereas $\Gb21=1+q^2$.
Both binomials are zero unless $0\leq b\leq a$.
Empty products have value one. Specializations are taken in the
polynomial or Laurent-polynomial forms of these expressions.

For $r,k\geq0$, put
\begin{equation}\label{eq:row-kernel}
 Z^+_{r,k}=\qb rk\prod_{i=0}^{k-1}D_{r+i},\qquad
 Z^-_{r,k}=\qb rk\prod_{i=0}^{k-1}D_{-r-i}.
\end{equation}
The kernels are zero for $k>r$ and equal one for $k=0$.

The first theorem gives a Laurent-integral Newton expansion in the width at each fixed height.

\begin{theorem}\label{thm:row-main}
For every zero-framed knot $K$ and every $s\geq1$, there is a unique sequence $C_{k,s}(K;A,q)\in \R$, $k\geq0$, such that
\begin{equation}\label{eq:row-expansion}
 H_{r,s}=\sum_{k=0}^r Z^+_{r,k}C_{k,s}\qquad(r\geq0).
\end{equation}
It satisfies $C_{0,s}=1$ and
\begin{equation}\label{eq:row-parity-factor}
 C_{k,s}(-A,q)=(-1)^kC_{k,s}(A,q),\qquad
 C_{k,s}\in D_{-s}\R\quad(k\geq1).
\end{equation}
\end{theorem}

Triangularity determines the coefficients over $\Q(A,q)$.
The assertion that they lie in $\R$ gives more: after the lower-width
terms have been subtracted, the remainder is divisible by every
factor of $Z^+_{k,k}$. The factor $D_{-s}$ in
\eqref{eq:row-parity-factor} records the specialization at which
$(r^s)$ becomes a determinant power.

To state the partition-indexed expansion, we use the following
notation. For a partition $\lambda$, write $|\lambda|=\sum_i\lambda_i$ and let $\lambda^T$ be its transpose. For a box $\square=(i,j)$, put
\[
 c(\square)=j-i,\qquad
 h(\square)=\lambda_i-j+\lambda_j^T-i+1,\qquad
 \ct(\lambda)=\sum_{\square\in\lambda}c(\square).
\]
For any partition $\lambda$, define
\begin{equation}\label{eq:small-dimension}
 D_\lambda^{(s)}(q)
 =s_\lambda(q^{s-1},q^{s-3},\ldots,q^{1-s})
 =\prod_{\square\in\lambda}
 \frac{[s+c(\square)]_q}{[h(\square)]_q}
 \quad (\ell(\lambda)\leq s),
\end{equation}
and set $D_\lambda^{(s)}=0$ for $\ell(\lambda)>s$.
Here $s_\lambda$ is the Schur polynomial. In particular,
$D_\lambda^{(s)}(q^{-1})=D_\lambda^{(s)}(q)$.

For a partition $\lambda$, the shape kernel is
\begin{equation}\label{eq:shape-kernel}
 \Phi_{\lambda,r}(A,q)=D_{\lambda^T}^{(r)}(q^{-1})
             \prod_{\square\in\lambda}D_{r+c(\square)}.
\end{equation}
It vanishes for $\lambda_1>r$, and
$\Phi_{(k),r}=Z^+_{r,k}$.

Since a row gives exactly $Z^+_{r,k}$,
Theorem~\ref{thm:row-main} supplies the coefficients required by
\cite[Conjecture~2.1]{KNTZ}.

\begin{corollary}\label{cor:kntz}
For every zero-framed knot $K$, set $\mathcal C_{\varnothing,s}=1$ and
\begin{equation}\label{eq:row-lift}
 \mathcal C_{\lambda,s}=
 \begin{cases}
 C_{k,s},&\lambda=(k),\ k\geq1,\\
 0,&\ell(\lambda)\geq2.
 \end{cases}
\end{equation}
Then $\mathcal C_{\lambda,s}\in \R$, it is independent of $r$, and it vanishes when $s<\ell(\lambda)$. For all $r,s\geq1$,
\begin{equation}\label{eq:KNTZ-general}
 H_{r,s}=\sum_{\lambda\subseteq(r^s)}
 \mathcal C_{\lambda,s}\Phi_{\lambda,r}.
\end{equation}
\end{corollary}

The monomial prefactor allowed in \cite[equation~(2.6)]{KNTZ} is one
in this normalization.

\Needspace{14\baselineskip}
A second interpolation in the height produces coefficients independent
of both dimensions. The two kernels in \eqref{eq:row-kernel} are chosen
so that transposition exchanges their roles.

\begin{theorem}\label{thm:double-main}
For every zero-framed knot $K$, there is a unique family
$(E_{k,l})_{k,l\geq0}$ of elements of $\R$ such that
\begin{equation}\label{eq:double-expansion}
 H_{r,s}=\sum_{k=0}^r\sum_{l=0}^s
 Z^+_{r,k}Z^-_{s,l}E_{k,l}\qquad(r,s\geq0).
\end{equation}
Its boundary values are $E_{0,0}=1$ and $E_{k,0}=E_{0,k}=0$ for $k>0$. Moreover,
\begin{align}
 E_{k,l}(A,q)&=E_{l,k}(A,q^{-1}),\label{eq:double-transpose}\\
 E_{k,l}(-A,q)&=(-1)^{k+l}E_{k,l}(A,q),\label{eq:double-A-parity}\\
 E_{k,l}(A,-q)&=(-1)^{k(k+1)/2+l(l+1)/2}E_{k,l}(A,q).
 \label{eq:double-q-parity}
\end{align}
\end{theorem}

The array $(E_{k,l})_{k,l\geq0}$ determines all rectangular colored
invariants of $K$ through \eqref{eq:double-expansion}. The values with
$r\leq k$ and $s\leq l$ determine $E_{k,l}$ by the finite recursion
\eqref{eq:double-recursion}.
At the first nontrivial color, the formula reads
\[
 H_{1,1}=1+D_1D_{-1}E_{1,1}.
\]
Higher $E_{k,l}$ describe the part of a rectangular invariant that
remains after contributions from smaller widths and heights have
been removed.

\subsection{The double-twist formula}
Let $K_{m,n}$ be the double-twist knot in the convention of
\cite[Figure~1]{KNTZ}: $K_{1,1}$ is the left-handed trefoil and
$K_{-1,1}$ is the figure-eight knot, as in \cite[Section~3]{CLZ}.
Figure~\ref{fig:plat} specifies the alternating plat used in the proof.

In this family, the knot-dependent coefficients can be given
explicitly. The color factors are the two-sided kernels
\begin{equation}\label{eq:two-sided-shape}
 \begin{split}
 \Psi_{\lambda;r,s}(A,q)
 &=\Phi_{\lambda,r}(A,q)\Phi_{\lambda^T,s}(A,q^{-1})\\
 &=D_{\lambda^T}^{(r)}(q^{-1})D_\lambda^{(s)}(q)
   \prod_{\square\in\lambda}D_{r+c(\square)}D_{c(\square)-s}.
 \end{split}
\end{equation}

Fix $r,s\geq1$ and let $\cP_{r,s}$ denote the partitions contained
in $(r^s)$, padded by zeroes to length $s$. For a box
$\square=(i,j)\in\lambda$, its arm length and leg length are
\[
 a_\lambda(\square)=\lambda_i-j,\qquad
 l_\lambda(\square)=\lambda_j^T-i.
\]
They count the boxes strictly to the right of $\square$ in its row
and strictly below it in its column, respectively; in particular,
$h(\square)=a_\lambda(\square)+l_\lambda(\square)+1$.
Put
\begin{equation}\label{eq:theta-shape}
 \Theta_\lambda=A^{2|\lambda|}q^{4\ct(\lambda)},\qquad
 t_\lambda=(-A^2)^{|\lambda|}q^{2\ct(\lambda)}.
\end{equation}
The KNTZ matrix on $\cP_{r,s}$ is
\begin{equation}\label{eq:B}
 B^{(s)}_{\lambda\mu}
 =q^{(1-s)(|\lambda|-|\mu|)}\Theta_\mu
 \det\left(\Gb{\lambda_i+s-i}{\mu_j+s-j}\right)_{i,j=1}^s.
\end{equation}
It vanishes unless $\mu\subseteq\lambda$, and its diagonal is $\Theta_\lambda$. Thus it is invertible over $\Q(A,q)$. Proposition~\ref{prop:compound} gives
\begin{equation}\label{eq:B-first}
 B^{(s)}_{\lambda\vac}=D_\lambda^{(s)}(q).
\end{equation}
For $m>0$, the twist factors of \cite[equations (2.2)--(2.4)]{KNTZ} are
\begin{equation}\label{eq:kntz-twists}
 F_\lambda^{(m;s)}=\frac{((B^{(s)})^m)_{\lambda\vac}}{D_\lambda^{(s)}},\qquad
 \Tw_\lambda^{(m)}=t_\lambda F_\lambda^{(m;s)},\qquad
 \Tw_\lambda^{(-m)}=F_\lambda^{(m;s)}(A^{-1},q^{-1}).
\end{equation}
The total arm and leg lengths satisfy
\[
 \sum_{\square\in\lambda}a_\lambda(\square)
 =\sum_i\binom{\lambda_i}{2},\qquad
 \sum_{\square\in\lambda}l_\lambda(\square)
 =\sum_j\binom{\lambda_j^T}{2}.
\]
Their difference is $\ct(\lambda)$, so the arm--leg exponent in
\cite[equation~(2.4)]{KNTZ} gives the factor
$q^{2\ct(\lambda)}$ in $t_\lambda$.
Proposition~\ref{prop:height} shows that the normalized first-column ratio is independent of $s\geq\max\{1,\ell(\lambda)\}$ and of the containing width. In particular,
\[
 \Tw_\lambda^{(1)}=t_\lambda,\qquad \Tw_\lambda^{(-1)}=1.
\]

Set
\begin{equation}\label{eq:twist-product}
 G_\lambda^{m,n}(A,q)=
 \frac{\Tw_\lambda^{(m)}\Tw_\lambda^{(n)}}
 {\Tw_\lambda^{(1)}\Tw_\lambda^{(-1)}}
 \qquad(m,n\neq0).
\end{equation}
Thus $G_\lambda^{m,n}$ is a rational function attached to the partition
and the two twist parameters, independent of the containing rectangle,
with $G_{\varnothing}^{m,n}=1$. The double-twist formula of
\cite[equation~(2.5)]{KNTZ} is as follows.

\begin{theorem}\label{thm:double-twist}
For every $r,s\geq1$ and every $m,n\in\Z\setminus\{0\}$,
\begin{equation}\label{eq:main}
\begin{split}
 H_{(r^s)}(K_{m,n};A,q)=
 \sum_{\lambda\subseteq(r^s)}
 &\frac{\Tw_\lambda^{(m)}\Tw_\lambda^{(n)}}
 {\Tw_\lambda^{(1)}\Tw_\lambda^{(-1)}}
 D_\lambda^{(s)}(q)D_{\lambda^T}^{(r)}(q^{-1})\\
 &\qquad\cdot\prod_{\square\in\lambda}
 D_{r+c(\square)}D_{c(\square)-s}.
\end{split}
\end{equation}
This is an identity of rational functions in independent variables $A,q$.
\end{theorem}

Equivalently,
\[
 H_{r,s}(K_{m,n};A,q)
 =\sum_{\lambda\subseteq(r^s)}G_\lambda^{m,n}\Psi_{\lambda;r,s}.
\]
The dependence on the rectangle
is carried by $\Psi_{\lambda;r,s}$; the two twist parameters are
carried by the separate factors of $G_\lambda^{m,n}$.

\subsection{The coefficient systems}
The partition kernels satisfy relations. For example,
\[
 \Phi_{(1,1),r}=\Phi_{(2),r}+D_0\Phi_{(1),r}.
\]
Consequently, at height at least two, a partition-indexed expansion
has more than one coefficient family. Theorems~\ref{thm:compression} and
\ref{thm:all-solutions} describe this freedom by integral changes
to the row kernels:
\[
 \Phi_{\lambda,r}
 =\sum_{k=\lambda_1}^{|\lambda|}T_{\lambda,k}Z^+_{r,k},
 \qquad T_{\lambda,k}\in\R.
\]
The coefficients $T_{\lambda,k}$ depend only on the partition and
on $A,q$. All the partition-indexed Laurent expansions described
by Theorem~\ref{thm:all-solutions} give the same width and double
coefficients.

For double-twist knots, Theorem~\ref{thm:coefficient-bridge} gives
\[
 E_{k,l}(K_{m,n})=
 \sum_{\substack{\lambda\subseteq(k^l)\\
                  \max(k,l)\leq|\lambda|}}
 G_\lambda^{m,n}T_{\lambda,k}(A,q)
 T_{\lambda^T,l}(A,q^{-1}).
\]
Here the transitions are extended by zero outside their stated
ranges, with $T_{\varnothing,0}=1$ and
$T_{\varnothing,k}=0$ for $k>0$. Each sum is finite and belongs to
$\R$. Thus the explicit KNTZ twist factors determine the general
coefficient array, with $E_{k,l}$ involving only partitions inside
its own $k\times l$ rectangle.

The coefficients also respect operations and specializations of
knots. The row kernels form an algebra over $\R$, whose structure
constants give the connected-sum formula of
Corollary~\ref{cor:connected-sum}. At $q=1$,
Theorem~\ref{thm:special-double} expresses every double coefficient
in terms of $E_{1,1}(K;A,1)$ by counting bipartite graphs without
isolated vertices.

\subsection{Proof strategy}
The general expansion combines Laurent integrality and the
sign and transpose symmetries of colored HOMFLY--PT invariants
\cite{Morton,Zhu} with representation symmetries at finite rank.
At $A=q^M$, the type-$A$ diagram involution identifies the colors
$r\omega_s$ and $r\omega_{M-s}$. Together with transposition, it
also gives a width reflection at $A=q^{-M}$. At the $k$th step of
the width expansion, this reflection reduces the remainder to
already determined smaller colors. The remainder therefore vanishes
at $A=\pm q^{-M}$ for $k\leq M\leq2k-1$. These roots give all the
prime factors of the diagonal kernel, so unique factorization
proves that the recursive division is exact in $\R$.
Positive-rank reflection passes to the coefficients $C_{k,s}$ and
allows the same argument in the height. This proves
Theorem~\ref{thm:double-main}. The integral shape transitions follow
by applying the width argument to the hook-content kernels.

The double-twist evaluation involves two adjacent fusion channels
and the pairing that closes the braid. We compare these data in
stable rank $N\geq2s$. For $R=\mathbb S_{(r^s)}\C^N$, a matrix
cross-ratio realizes
\[
 \Inv_{\gl_N}(R\otimes R^*\otimes R\otimes R^*)
 \simeq
 \{f\in\C[x_1,\ldots,x_s]^{\mathfrak S_s}:
                  \deg_{x_i}f\leq r\text{ for every }i\}.
\]
Multiplication by the Vandermonde identifies the right side with
$\bigwedge^s\C[t]_{\leq r+s-1}$. Under this same identification,
both adjacent Casimirs become additive exterior powers of one-row
operators at rank $N'=N-2s+2$, with a common scalar shift
(Theorem~\ref{thm:two-residues}).

The Drinfeld--Kohno theorem transports this comparison to quantum
pure braids through KZ monodromy \cite{Drinfeld,Kohno,Brochier}.
Vacuum normalization fixes the scalar factors of the transported
operators. The cup and cap pairings are expressed by spectral
projectors, giving a trace ratio invariant under simultaneous
conjugation. This preserves the reduced closure. The rectangular knot
invariant is thereby expressed as an initial principal minor of a word
in one-row matrices; see Theorem~\ref{thm:rt-determinant}.

On the formula side, the Gaussian minors in $B^{(s)}$ have the same
exterior-power structure. Cauchy--Binet turns the KNTZ shape sum into a second determinant. The one-row identity of Chen, Liu,
and Zhu identifies the vacuum moments of the two matrix systems;
Lemma~\ref{lem:moment-lift} turns this identity into an
upper-triangular congruence, which identifies their initial principal
minors. A mirror-inverse identity treats negative twists. Finally,
the equalities for infinitely many stable ranks $A=q^N$ give the
identity in independent variables $A,q$.

Sections~\ref{sec:rank} and \ref{sec:double} prove the integral
expansions and the shape transitions.
Section~\ref{sec:classical} establishes the classical two-operator
comparison, Section~\ref{sec:closure} passes to the reduced quantum
closure, and Section~\ref{sec:double-twist-proof} proves the double-twist formula.
Section~\ref{sec:coefficients} treats the coefficient comparison,
connected sums, and $q=1$.
Section~\ref{sec:consequences} gives examples and the obstruction to
extra two-sided factors. The appendices contain further Gaussian
identities and the two-strand Rosso--Jones formula in our convention.

\subsection{Further questions}
Proposition~\ref{prop:obstruction} gives a concrete obstruction to
extracting all two-sided box factors: $D_0$ does not divide $E_{2,1}$
for $T(3,4)$. A first problem is to determine which knots admit
additional factors in specified coefficients and how these factors are
governed by the Alexander polynomial and the defect studied in
\cite{KM,LM}. The fixed kernels in Theorem~\ref{thm:double-main}
formulate this problem independently of a choice of partition
coefficients.

A second problem is to compare, after the specialization $A=q^N$, the
double expansion on admissible rectangular colors with the central
interpolation expansion of Beliakova and Gorsky \cite{BG}. An explicit
relation between the two coefficient systems would identify the
fixed-rank data detected by rectangular colors and describe how
two-variable integrality behaves under rank specialization. The
HOMFLY--PT limit of their interpolation polynomials provides a natural
point of comparison \cite[Remark~11.7]{BG}.

A third direction is to identify additional structure that selects a
partition-indexed coefficient family from the freedom described in
Theorem~\ref{thm:all-solutions}. The prescribed rational twist factors
for double-twist knots provide a concrete test case. The refined and
homological formulas in \cite[Section~3]{KNTZ} suggest graded analogues
of the finite coefficient identities. A related problem is to extend
the simultaneous channel comparison to a graded setting while preserving
the closure pairing.

\Needspace{12\baselineskip}
\section{Rank symmetries and integral interpolation}\label{sec:rank}
Morton's theorem on meridian-eigenvector $(1,1)$-tangle invariants gives
\begin{equation}\label{eq:standard-integrality}
 H_\lambda(K;A,q)\in \R
\end{equation}
for every partition $\lambda$ \cite[Theorem~1]{Morton}. The sign and transpose properties in the reduced normalization are
\begin{align}
 H_\lambda(K;-A,q)&=H_\lambda(K;A,q),\label{eq:A-even}\\
 H_\lambda(K;A,-q)&=H_\lambda(K;A,q),\label{eq:q-even}\\
 H_\lambda(K;A,q^{-1})&=H_{\lambda^T}(K;A,q).
 \label{eq:transpose}
\end{align}
These follow from \cite[Theorem~1.1, Corollary~1.3, and equations~(1.13)--(1.15)]{Zhu}; the signs on the colored knot and colored unknot cancel after reduction. In Morton's variables one takes $v=A^{-1}$ and $s=q$; in Zhu's notation $a=A$.

\subsection{Finite-rank reflections}
At $A=q^M$, with $\ell(\lambda)\leq M$, the colored invariant is the reduced quantum-group invariant in color $\lambda$ \cite{RT,Zhu}. For a zero-framed knot, its scalar depends only on the corresponding $\mathfrak{sl}_M$ highest weight. Indeed, the central factor in the passage from $\mathfrak{gl}_M$ to $\mathfrak{sl}_M$ contributes only the abelian self-linking factor; its exponent is zero. In particular, tensoring the color by a determinant power leaves the reduced scalar unchanged.

Let $\omega_1,\ldots,\omega_{M-1}$ be the fundamental weights of
$\slie_M$, and let $\rho$ be the half-sum of its positive roots.

\begin{lemma}\label{lem:diagram}
For $M\geq2$, the type-$A_{M-1}$ diagram involution preserves the reduced quantum invariant of every zero-framed knot. In particular, the invariants in colors $r\omega_s$ and $r\omega_{M-s}$ are equal for $1\leq s<M$.
\end{lemma}
\begin{proof}
On the Drinfeld--Jimbo generators, the map
\[
 E_i\longmapsto E_{M-i},\qquad
 F_i\longmapsto F_{M-i},\qquad
 K_i\longmapsto K_{M-i}
\]
preserves the relations and the coproduct. Here $K_i$ denotes the simple-root Cartan generator. The diagram involution preserves the invariant bilinear form and the positive-root system. It therefore preserves the Cartan part of the universal $R$-matrix. The root part is defined by the canonical pairing of the positive and negative Borel parts; the involution preserves that pairing and its canonical tensor. Consequently it fixes the universal $R$-matrix. It fixes $\rho$, and hence the pivotal element, and it fixes the ribbon element determined by these data.

Twisting a module by this automorphism thus gives a ribbon equivalence preserving every elementary crossing, evaluation, and coevaluation in the quantum tangle functor. The closed evaluations and the corresponding unknot evaluations agree. Their quotient is the reduced invariant. The involution sends $\omega_s$ to $\omega_{M-s}$, proving the assertion. 
\end{proof}
These constructions may be taken in the standard $h$-adic realization of the generic ribbon category \cite{RT}; equality there gives equality of the generic rational functions.

\begin{proposition}\label{prop:reflections}
For every knot $K$, the rectangular array satisfies
\begin{align}
 H_{r,s}(q^M,q)&=H_{r,M-s}(q^M,q)
 &&(M\geq1,\ r\geq0,\ 0\leq s\leq M),\label{eq:positive-reflection}\\
 H_{r,s}(q^{-M},q)&=H_{M-r,s}(q^{-M},q)
 &&(M\geq1,\ s\geq0,\ 0\leq r\leq M).
 \label{eq:negative-reflection}
\end{align}
In particular, $H_{r,s}(q^s,q)=1$ for $s\geq1$.
\end{proposition}
\begin{proof}
For $M\geq2$ and $0<s<M$, the partition $(r^s)$ represents the $\mathfrak{sl}_M$ highest weight $r\omega_s$. Lemma~\ref{lem:diagram} identifies its reduced scalar with that of $(r^{M-s})$. At $s=0$ or $s=M$, the colors are the empty partition and a determinant power, and both scalars are one. This proves \eqref{eq:positive-reflection} for $M\geq2$.

At rank one, the color $(r)$ is a one-dimensional $U_q(\mathfrak{gl}_1)$ module. The root part of the universal matrix is trivial. Its abelian Cartan and ribbon factors give a power determined by the framing, and that power is one for framing zero. Thus $H_{r,1}(q,q)=1$, proving the rank-one case as well.

For the second identity, put $\bar q=q^{-1}$. By transposition and the first identity,
\[
 \begin{aligned}
 H_{r,s}(q^{-M},q)
 &=H_{s,r}(\bar q^M,\bar q)\\
 &=H_{s,M-r}(\bar q^M,\bar q)\\
 &=H_{M-r,s}(q^{-M},q).
 \end{aligned}
\]
The empty-color cases follow from the same calculation. The substitutions are well defined in the Laurent polynomial ring by \eqref{eq:standard-integrality}.
\end{proof}

\subsection{Laurent interpolation in the width}\label{sec:interpolation}
The following argument uses only Laurent polynomials and the kernels \eqref{eq:row-kernel}.

\begin{lemma}\label{lem:collision}
Let $k\geq1$, $0\leq t<k$, and $M=k+t$. At $A=q^{-M}$,
\begin{align}
 Z^+_{k,j}&=Z^+_{t,j}&&(0\leq j\leq t),\label{eq:collision-agreement}\\
 Z^+_{k,j}&=0&&(t<j\leq k).\label{eq:collision-vanishing}
\end{align}
\end{lemma}
\begin{proof}
At this specialization, $D_{k+i}=\br{i-t}$. If $j>t$, the product defining $Z^+_{k,j}$ contains $\br{0}$. For $j\leq t$,
\[
 Z^+_{k,j}(q^{-M},q)
 =(-1)^j\frac{\br{k}!\br{t}!}
 {\br{j}!\br{k-j}!\br{t-j}!}.
\]
Interchanging $k$ and $t$ in this expression gives $Z^+_{t,j}(q^{-M},q)$.
\end{proof}

\begin{theorem}\label{thm:abstract}
Let $F_r\in \R$, $r\geq0$, and let $\eps\in\{1,-1\}$. Suppose
\begin{align}
 F_r(-A,q)&=\eps F_r(A,q),\label{eq:abstract-parity}\\
 F_r(q^{-M},q)&=F_{M-r}(q^{-M},q)
 &&(M\geq1,\ 0\leq r\leq M).
 \label{eq:abstract-reflection}
\end{align}
Then there is a unique sequence $c_k\in \R$ satisfying
\begin{equation}\label{eq:abstract-expansion}
 F_r=\sum_{k=0}^r Z^+_{r,k}c_k\qquad(r\geq0).
\end{equation}
It has $c_0=F_0$ and $c_k(-A,q)=\eps(-1)^kc_k(A,q)$.
\end{theorem}
\begin{proof}
The diagonal kernel is
\begin{equation}\label{eq:diagonal}
 Z^+_{k,k}=\prod_{M=k}^{2k-1}D_M\neq0.
\end{equation}
Thus the coefficients are uniquely determined over $\Q(A,q)$ by triangular recursion. We prove their integrality and parity by induction.

Set $c_0=F_0$. Suppose $c_0,\ldots,c_{k-1}$ have been constructed in $\R$ with the stated parity, and put
\[
 R_k=F_k-\sum_{j=0}^{k-1}Z^+_{k,j}c_j\in \R.
\]
Since $Z^+_{k,j}$ has $A$-parity $(-1)^j$, every term in this expression has parity $\eps$. Fix $t\in\{0,\ldots,k-1\}$ and put $M=k+t$. By Lemma~\ref{lem:collision} and \eqref{eq:abstract-reflection},
\[
 \begin{aligned}
 R_k(q^{-M},q)
 &=F_t(q^{-M},q)
   -\sum_{j=0}^tZ^+_{t,j}(q^{-M},q)c_j(q^{-M},q)\\
 &=0.
 \end{aligned}
\]
The induction hypothesis gives the last equality. Parity then implies $R_k(-q^{-M},q)=0$.

The ring $\R$ is a UFD. For each $M$, the elements $A-q^{-M}$ and $A+q^{-M}$ are prime: the quotient by either element is $\Z[q^{\pm1}]$. As $M$ ranges from $k$ to $2k-1$, these primes are pairwise nonassociate. Therefore
\[
 \prod_{M=k}^{2k-1}(A-q^{-M})(A+q^{-M})\mid R_k.
\]
For each $M$,
\[
 D_M=A^{-1}q^M(A-q^{-M})(A+q^{-M}).
\]
The product differs from \eqref{eq:diagonal} by a Laurent unit. It follows that
\[
 c_k=R_k/Z^+_{k,k}\in \R.
\]
The parity of this quotient is $\eps(-1)^k$, completing the induction.
\end{proof}

\begin{proof}[Proof of Theorem~\ref{thm:row-main}]
For fixed $s$, apply Theorem~\ref{thm:abstract} with $F_r=H_{r,s}$ and $\eps=1$. Its hypotheses follow from \eqref{eq:A-even} and Proposition~\ref{prop:reflections}. This proves \eqref{eq:row-expansion}, uniqueness, and parity.

At $A=q^s$, the array has value $H_{r,s}=1$. Every specialized diagonal kernel is nonzero in $\Q(q)$, since
\[
 Z^+_{k,k}(q^s,q)=\prod_{M=k}^{2k-1}\br{s+M}\neq0.
\]
Triangular uniqueness gives $C_{k,s}(q^s,q)=0$ for $k>0$. Parity supplies the root $A=-q^s$ as well. It follows that
$D_{-s}=A^{-1}q^{-s}(A-q^s)(A+q^s)$ divides $C_{k,s}$ in $\R$.
\end{proof}

\begin{proof}[Proof of Corollary~\ref{cor:kntz}]
For a row $(k)$, the hook formula gives
\[
 D_{(1^k)}^{(r)}(q^{-1})
 =\prod_{i=0}^{k-1}\frac{\br{r-i}}{\br{k-i}}
 =\qb rk.
\]
The contents of $(k)$ are $0,\ldots,k-1$, so $\Phi_{(k),r}=Z^+_{r,k}$. The row-supported sum \eqref{eq:KNTZ-general} is therefore \eqref{eq:row-expansion}. The coefficient ring, empty-shape value, height condition, and width independence are immediate from \eqref{eq:row-lift} and Theorem~\ref{thm:row-main}.
\end{proof}

The construction is given by the finite recursion
\begin{equation}\label{eq:row-recursion}
 C_{k,s}=\frac{H_{k,s}-1-\sum_{j=1}^{k-1}Z^+_{k,j}C_{j,s}}
 {\prod_{M=k}^{2k-1}D_M}\qquad(k\geq1).
\end{equation}
Every quotient in \eqref{eq:row-recursion} is exact in $\R$. In particular,
\[
 C_{1,s}=\frac{H_{1,s}-1}{D_1},\qquad
 C_{2,s}=\frac{H_{2,s}-1-(q+q^{-1})D_2C_{1,s}}{D_2D_3}.
\]
These coefficients use only the rectangular invariants of widths at most $k$ at the fixed height $s$.

\begin{corollary}\label{cor:rank-congruences}
For $r\neq t$, $r,t\geq0$, and $s\geq1$,
\[
 D_{r+t}D_{-s}\mid H_{r,s}-H_{t,s}
 \quad\text{in }\R.
\]
\end{corollary}
\begin{proof}
Negative-rank reflection at $M=r+t$ and $A$-evenness give the factor $D_{r+t}$. Determinant-rank specialization and evenness give $D_{-s}$. Their prime factors are distinct because $r+t\geq1$ and $s\geq1$, so their product divides the difference.
\end{proof}

\section{Two-dimensional coefficients and shape kernels}\label{sec:double}
\subsection{Interpolation in the height}
For $s=0$, put $C_{0,0}=1$ and $C_{k,0}=0$ for $k>0$, consistently with the expansion of $H_{r,0}=1$.

\begin{proposition}\label{prop:coefficient-reflection}
For $M\geq1$, $0\leq s\leq M$, and $k\geq0$,
\begin{equation}\label{eq:height-coefficient-reflection}
 C_{k,s}(q^M,q)=C_{k,M-s}(q^M,q).
\end{equation}
For $s,t\geq0$ with $s+t\geq1$ and $k\geq0$,
\begin{equation}\label{eq:height-coefficient-divisibility}
 D_{-(s+t)}\mid C_{k,s}-C_{k,t}.
\end{equation}
\end{proposition}
\begin{proof}
Specialize the row expansions for heights $s$ and $M-s$ at $A=q^M$. Their left sides agree for every width by \eqref{eq:positive-reflection}. Their diagonal kernels are
\[
 Z^+_{k,k}(q^M,q)=\prod_{a=k}^{2k-1}\br{M+a}\neq0
\]
for $k>0$, and equal one for $k=0$. Triangular uniqueness over $\Q(q)$ proves \eqref{eq:height-coefficient-reflection}. Apply this with $M=s+t$. The two coefficients have the same $A$-parity, so the difference vanishes at $A=\pm q^M$. The two associated prime factors give \eqref{eq:height-coefficient-divisibility}.
\end{proof}

\begin{proof}[Proof of Theorem~\ref{thm:double-main}]
Fix $k\geq0$ and consider the sequence $s\mapsto C_{k,s}(A,q)$. It has parity $(-1)^k$ in $A$ and satisfies \eqref{eq:height-coefficient-reflection}. Apply Theorem~\ref{thm:abstract} after the substitution $q\mapsto q^{-1}$. Since the symmetric quantum binomial is unchanged by this substitution, the resulting kernel is $Z^-_{s,l}$. Thus there are unique $E_{k,l}\in \R$ such that
\begin{equation}\label{eq:C-from-E}
 C_{k,s}=\sum_{l=0}^s Z^-_{s,l}E_{k,l}.
\end{equation}
Its constant coefficient is $E_{k,0}=C_{k,0}$, and its parity is $(-1)^{k+l}$. When $k=0$, the sequence is constantly one, so $E_{0,l}=0$ for $l>0$. Substitution into the row expansion proves \eqref{eq:double-expansion}. Uniqueness follows by successive triangular uniqueness in the two variables.

Transposition gives $H_{r,s}(A,q)=H_{s,r}(A,q^{-1})$. Inverting $q$ interchanges $Z^+$ and $Z^-$, so uniqueness of the double expansion proves \eqref{eq:double-transpose}.

Finally,
\[
 \qb rk\Big|_{q\mapsto-q}=(-1)^{k(r-k)}\qb rk,
 \qquad
 D_{\pm(r+i)}(A,-q)=(-1)^{r+i}D_{\pm(r+i)}(A,q).
\]
Consequently, with $a_k=k(k+1)/2$,
\begin{equation}\label{eq:kernel-q-parity}
 Z^\pm_{r,k}(A,-q)=(-1)^{a_k}Z^\pm_{r,k}(A,q).
\end{equation}
The left side of \eqref{eq:double-expansion} is even in $q$ by \eqref{eq:q-even}. A final application of uniqueness proves \eqref{eq:double-q-parity}.
\end{proof}

For $k,l>0$, the coefficients can be obtained directly in increasing order of $k+l$:
\begin{equation}\label{eq:double-recursion}
 E_{k,l}=\frac{\displaystyle H_{k,l}
 -\mathop{\sum_{0\leq i\leq k,\ 0\leq j\leq l}}_{(i,j)\neq(k,l)}
 Z^+_{k,i}Z^-_{l,j}E_{i,j}}
 {\displaystyle\left(\prod_{a=k}^{2k-1}D_a\right)
  \left(\prod_{b=l}^{2l-1}D_{-b}\right)}.
\end{equation}
Theorem~\ref{thm:double-main} proves that these are exact divisions in $\R$. The first coefficient is
\begin{equation}\label{eq:E11}
 E_{1,1}=\frac{H_{1,1}-1}{D_1D_{-1}}.
\end{equation}
For example,
\begin{equation}\label{eq:22-expansion}
 \begin{split}
 H_{2,2}=1
 &+(q+q^{-1})^2D_2D_{-2}E_{1,1}\\
 &+(q+q^{-1})D_2D_{-2}D_{-3}E_{1,2}\\
 &+(q+q^{-1})D_2D_3D_{-2}E_{2,1}\\
 &+D_2D_3D_{-2}D_{-3}E_{2,2}.
 \end{split}
\end{equation}
Every identity above is an identity in $\R$. It therefore specializes at any nonzero values of $A$ and $q$, including roots of unity, by evaluating the Laurent-polynomial kernels and coefficients. Triangular uniqueness is asserted over $\R$; diagonal kernels can vanish after specialization.

\subsection{Integral shape transitions}\label{sec:shapes}
Set
\[
 X_t=Aq^{2t}+A^{-1}q^{-2t},\qquad
 \Pi_k=\br{k}!,\qquad
 R_k(X)=\frac{\prod_{i=0}^{k-1}(X-X_i)}{\Pi_k}.
\]
The polynomials $R_k$ belong to $\Q(A,q)[X]$. Their values at $X=X_r$ are the integral kernels $Z^+_{r,k}$.

\begin{lemma}\label{lem:content}
For every partition $\lambda$ and every $r\geq0$,
\begin{equation}\label{eq:content-identity}
 \Phi_{\lambda,r}
 =\frac{\prod_{\square\in\lambda}(X_r-X_{c(\square)})}
 {\prod_{\square\in\lambda}\br{h(\square)}}.
\end{equation}
\end{lemma}
\begin{proof}
The hook-content formula \cite[Chapter~I, Section~3]{Macdonald} and transposition give
\[
 D_{\lambda^T}^{(r)}(q^{-1})
 =\prod_{\square\in\lambda}
 \frac{\br{r-c(\square)}}{\br{h(\square)}}.
\]
The elementary identity
\[
 X_r-X_c=\br{r-c}D_{r+c}
\]
then proves \eqref{eq:content-identity}. When $r<\lambda_1$, both sides vanish because the first row contains a box of content $r$.
\end{proof}

\begin{theorem}\label{thm:compression}
Let $d=|\lambda|$ and $w=\lambda_1$ for a nonempty partition. There are unique Laurent polynomials $T_{\lambda,k}\in \R$, $w\leq k\leq d$, such that
\begin{equation}\label{eq:compression}
 \Phi_{\lambda,r}=\sum_{k=w}^dT_{\lambda,k}Z^+_{r,k}
 \qquad(r\geq0).
\end{equation}
They have $A$-parity $(-1)^{d-k}$. Their boundary values are
\begin{align}
 T_{\lambda,d}
 &=\frac{\Pi_d}{\prod_{\square\in\lambda}\br{h(\square)}},\label{eq:top-transition}\\
 T_{\lambda,w}
 &=D_{\lambda^T}^{(w)}(q^{-1})
   \prod_{\substack{(i,j)\in\lambda\\i\geq2}}D_{w+j-i}.
 \label{eq:bottom-transition}
\end{align}
\end{theorem}
\begin{proof}
The right side of \eqref{eq:content-identity}, with $X_r$ replaced by $X$, is a polynomial of degree $d$ over $\Q(A,q)$. The polynomials $R_0,\ldots,R_d$ form a basis of that polynomial space. Since the first row of $\lambda$ has contents $0,\ldots,w-1$, its expansion contains only $R_w,\ldots,R_d$. This gives unique rational coefficients $T_{\lambda,k}$.

For every $r\geq0$, the value $\Phi_{\lambda,r}$ is in $\R$: its principal specialization is a Schur polynomial evaluated at Laurent monomials, and the remaining factors are in $\R$. Its $A$-parity is $(-1)^d$. At $A=q^{-M}$ one has $X_r=X_{M-r}$, so the sequence $r\mapsto\Phi_{\lambda,r}$ satisfies the reflection hypothesis of Theorem~\ref{thm:abstract}. That theorem gives integral interpolation coefficients with the stated parity. By triangular uniqueness, they agree with the finite rational expansion already constructed. This proves integrality.

Comparison of leading coefficients gives \eqref{eq:top-transition}. Evaluating \eqref{eq:compression} at $r=w$ leaves only its $k=w$ term. Cancelling the factors of $Z^+_{w,w}$ corresponding to the first row of $\lambda$ proves \eqref{eq:bottom-transition}.
\end{proof}

\begin{example}\label{ex:shape-ambiguity}
The first nontrivial relation is
\begin{equation}\label{eq:11-relation}
 \Phi_{(1,1),r}=Z^+_{r,2}+D_0Z^+_{r,1}.
\end{equation}
Indeed, the two numerators for $\Phi_{(1,1),r}$ and $Z^+_{r,2}$ differ by $(X_r-X_0)(X_1-X_{-1})$, and $X_1-X_{-1}=\br{2}D_0$.
\end{example}

\begin{theorem}\label{thm:all-solutions}
Fix a knot $K$ and $s\geq1$. For each multirow partition $\lambda$ with $2\leq\ell(\lambda)\leq s$, choose an arbitrary $a_{\lambda,s}\in \R$. Define
\begin{align}
 \widetilde{\mathcal C}_{\lambda,s}&=a_{\lambda,s}
 &&(2\leq\ell(\lambda)\leq s),\label{eq:shape-free}\\
 \widetilde{\mathcal C}_{(k),s}
 &=C_{k,s}-
 \sum_{\substack{2\leq\ell(\lambda)\leq s\\
                 \lambda_1\leq k\leq|\lambda|}}
 T_{\lambda,k}a_{\lambda,s}
 &&(k\geq1),\label{eq:shape-row-correction}
\end{align}
with $\widetilde{\mathcal C}_{\varnothing,s}=1$ and $\widetilde{\mathcal C}_{\lambda,s}=0$ for $\ell(\lambda)>s$. The sums in \eqref{eq:shape-row-correction} are finite. These formulas give all Laurent-polynomial coefficient families satisfying \eqref{eq:KNTZ-general} and the stated empty-shape and height conditions.
\end{theorem}
\begin{proof}
For fixed $k$ and $s$, every partition contributing to \eqref{eq:shape-row-correction} is contained in $(k^s)$, so the sum is finite. For fixed external width $r$, only partitions inside $(r^s)$ contribute to \eqref{eq:KNTZ-general}. Apply Theorem~\ref{thm:compression} to each such partition. The coefficient of $Z^+_{r,k}$, for $k\leq r$, is
\[
 \widetilde{\mathcal C}_{(k),s}
 +\sum_{\substack{2\leq\ell(\lambda)\leq s\\
                 \lambda_1\leq k\leq|\lambda|}}
 T_{\lambda,k}a_{\lambda,s}=C_{k,s}.
\]
Every partition in this sum is already contained in $(r^s)$. Terms with $k>r$ vanish. Thus the shape sum equals the row expansion.

Conversely, start with any Laurent coefficient family satisfying \eqref{eq:KNTZ-general}, and take its multirow entries as $a_{\lambda,s}$. Compression gives a row expansion with coefficients equal to the left side of the preceding display. Uniqueness in Theorem~\ref{thm:row-main} forces those coefficients to be $C_{k,s}$, which gives \eqref{eq:shape-row-correction}.
\end{proof}

For example, when $s\geq2$, relation \eqref{eq:11-relation} allows the simultaneous replacements
\[
 \mathcal C_{(1,1),s}\longmapsto\mathcal C_{(1,1),s}+f,\qquad
 \mathcal C_{(2),s}\longmapsto\mathcal C_{(2),s}-f,\qquad
 \mathcal C_{(1),s}\longmapsto\mathcal C_{(1),s}-D_0f
\]
for any $f\in \R$. The rectangular invariants are unchanged. The coefficients $E_{k,l}$ give a unique expansion in the two
rectangular dimensions. The coefficients selected by the double-twist
formula are related to this expansion in
Theorem~\ref{thm:coefficient-bridge}.

\section{Four-point Casimirs and exterior powers}\label{sec:classical}
\subsection{The polynomial model}\label{sec:space}
Fix $N\geq2s$, set $V=\C^N$, and let $R=\mathbb S_{(r^s)}V$. Consider
\begin{equation}\label{eq:invariant-space}
 \cV_{r,s,N}=\Inv_{\gl_N}(R\otimes R^*\otimes R\otimes R^*).
\end{equation}
Introduce $N\times s$ matrices of variables $X_1,X_2$ and $s\times N$ matrices $Y_1,Y_2$. The variables are representation generators: the $X$ variables transform as copies of $V$, and the $Y$ variables as copies of $V^*$. The Cauchy decomposition \cite[Chapter I, Sections 3--4 and Appendix A]{Macdonald}
\[
 \Sym(V\otimes\C^s)=\bigoplus_\lambda
 \mathbb S_\lambda V\otimes\mathbb S_\lambda\C^s
\]
shows that the $\det^r$ component for the auxiliary $\GL_s$ acting on $X_a$ is $R$. The corresponding component for $Y_a$ is $R^*$. Hence \eqref{eq:invariant-space} is the space of $\GL_N$-invariant polynomials carrying the determinant character $\det^r$ at each of the four auxiliary factors.

Write
\begin{equation}\label{eq:crossratio}
 M_{ab}=Y_aX_b,\qquad
 T=M_{11}^{-1}M_{12}M_{22}^{-1}M_{21}.
\end{equation}
The invariant ring has a direct description in these coordinates.

\begin{lemma}\label{lem:contractions}
For $N\geq2s$, the $\GL_N$-invariant polynomial ring in $X_1,X_2,Y_1,Y_2$ is the polynomial ring in the entries of the four matrices $M_{ab}$.
\end{lemma}
\begin{proof}
Let $k=2s$, combine the $X$ matrices into an $N\times k$ matrix $X$, and combine the $Y$ matrices into a $k\times N$ matrix $Y$. On the full-rank chart for $X$, a change of basis puts $X$ in the form $\binom{I_k}{0}$. Write $Y=(M\ Z)$. The stabilizer of this $X$ consists of block matrices
\[
 g=\begin{pmatrix}I_k&B\\0&C\end{pmatrix},
\]
which act on $Y$ by
\[
 (M\ Z)g^{-1}=(M\ (Z-MB)C^{-1}).
\]
On the dense locus where $M$ is invertible, choose $B=M^{-1}Z$ to set $Z=0$. An invariant polynomial therefore equals its restriction at $X=\binom{I_k}{0}$, $Y=(M\ 0)$, which is a polynomial in $M=YX$. The equality holds on a dense open set and hence everywhere. Conversely, all entries of $YX$ are invariant, and every $k\times k$ matrix occurs as $YX$ in this chart. They are therefore algebraically independent. When $N=k$, the same argument applies with the complementary blocks omitted.
\end{proof}

\begin{proposition}\label{prop:polynomial-model}
There is a linear isomorphism
\begin{equation}\label{eq:polynomial-model}
 \cV_{r,s,N}\simeq
 \{F\in\C[t_1,\ldots,t_s]^{\mathfrak S_s}:\deg_{t_i}F\leq r\text{ for all }i\},
\end{equation}
under which
\begin{equation}\label{eq:prefactor}
 f(X,Y)=(\det M_{11}\det M_{22})^rF(T).
\end{equation}
Here $F(T)$ is the conjugation-invariant polynomial whose eigenvalue expression is $F(t_1,\ldots,t_s)$.
\end{proposition}
\begin{proof}
The auxiliary actions send $M_{ab}$ to $g_aM_{ab}h_b$. On the chart
where $M_{11}$, $M_{22}$, and $M_{12}$ are invertible, these actions set
all three matrices to $I_s$. The residual action conjugates the remaining
matrix. Tracking the determinant character gives the prefactor in
\eqref{eq:prefactor}. Evaluation with the first three matrices equal to
$I_s$ shows that $F$ is a polynomial conjugation invariant, hence a
symmetric polynomial in the eigenvalues.

Set $M_{11}=\diag(u_1,\ldots,u_s)$, $M_{22}=M_{12}=I_s$, and $M_{21}=\diag(v_1,\ldots,v_s)$. Polynomiality of $f$ gives polynomiality of
\[
 \prod_i u_i^r F(v_1/u_1,\ldots,v_s/u_s).
\]
Distinct monomials in the $v_i$ cannot cancel. Therefore the exponent of each variable in $F$ is at most $r$.

Conversely, Schur polynomials $s_\lambda(t)$ with $\lambda\subseteq(r^s)$ form a basis of the space on the right of \eqref{eq:polynomial-model}. For such a $\lambda$, every entry of
\[
 (\det M)^r\mathbb S_\lambda(M^{-1})
\]
is polynomial in $M$: after transposition this is the polynomial representation with highest weight $(r-\lambda_s,\ldots,r-\lambda_1)$. Apply this observation to both inverse matrices in
\[
 (\det M_{11}\det M_{22})^r
 \Tr\!\left(\mathbb S_\lambda(M_{11}^{-1})
 \mathbb S_\lambda(M_{12})\mathbb S_\lambda(M_{22}^{-1})
 \mathbb S_\lambda(M_{21})\right).
\]
The expression is polynomial and has the required four determinant characters. It is \eqref{eq:prefactor} for $F=s_\lambda$. This proves surjectivity. Restriction to the open chart proves injectivity.
\end{proof}
In particular, $\dim\cV_{r,s,N}=\binom{r+s}{s}$. The vector $F=1$ is the invariant obtained by pairing factors $(1,2)$ and $(3,4)$ separately.

\subsection{The adjacent Casimir operators}\label{sec:casimirs}
On the first pair of factors, the infinitesimal action is
\begin{equation}\label{eq:generators}
 E_{ij}=\sum_{\alpha=1}^sX_{1,i\alpha}\frac{\partial}{\partial X_{1,j\alpha}}
 -\sum_{\alpha=1}^sY_{1,\alpha j}\frac{\partial}{\partial Y_{1,\alpha i}}.
\end{equation}
The minus sign is the dual action on the second factor. Let
\[
 C_{12}=\sum_{i,j=1}^N E_{ij}E_{ji}.
\]
Each $E_{ij}$ annihilates $M_{11}$ and $M_{22}$, so the prefactor in \eqref{eq:prefactor} is unaffected.

\begin{lemma}\label{lem:matrix-casimir}
On the cross-ratio chart,
\begin{align}
 C_{12}T_{ab}&=2NT_{ab}-2s\delta_{ab},\label{eq:drift}\\
 \sum_{i,j}(E_{ij}T_{ab})(E_{ji}T_{cd})
 &=2T_{ad}T_{cb}-\delta_{ad}T_{cb}-T_{ad}\delta_{cb}.\label{eq:covariance}
\end{align}
\end{lemma}
\begin{proof}
Use the column and row vectors
\[
 u_j=M_{11}^{-1}Y_{1,:,j},\quad
 v_j=M_{11}^{-1}M_{12}M_{22}^{-1}Y_{2,:,j},\quad
 a_i=X_{2,i,:}M_{22}^{-1}M_{21},\quad b_i=X_{1,i,:}.
\]
Differentiation of \eqref{eq:crossratio} gives
\[
 E_{ij}T=-u_ja_i+v_jb_i.
\]
The contractions
\[
 \sum_i u_i a_i=T,\qquad \sum_i v_i b_i=T,\qquad
 \sum_i u_i b_i=I_s,\qquad \sum_i v_i a_i=T
\]
give \eqref{eq:covariance} after expansion.
For the first-order term, $C_{12}M_{12}=NM_{12}$ and $C_{12}M_{21}=NM_{21}$. In the product rule, the mixed contraction is
\[
 -2\sum_iY_{1,:,i}X_{1,i,:}
       \sum_j X_{2,j,:}M_{22}^{-1}Y_{2,:,j}
 =-2sM_{11},
\]
because the second sum is $\Tr(M_{22}^{-1}Y_2X_2)=s$. Multiplication by $M_{11}^{-1}$ gives \eqref{eq:drift}.
\end{proof}

\begin{proposition}\label{prop:radial}
In the polynomial model \eqref{eq:polynomial-model}, the operator $C_{12}/2$ is
\begin{equation}\label{eq:radial}
 D_s=\sum_i t_i(t_i-1)\partial_i^2+
 \sum_i\left((N-2s+2)t_i-1+
 2t_i(t_i-1)\sum_{j\ne i}\frac1{t_i-t_j}\right)\partial_i.
\end{equation}
Define
\begin{equation}\label{eq:inversion}
 (P_sF)(t_1,\ldots,t_s)=\prod_i t_i^r F(t_1^{-1},\ldots,t_s^{-1}).
\end{equation}
Then the Casimir of factors $(2,3)$ satisfies
\begin{equation}\label{eq:second-radial}
 C_{23}/2=P_sD_sP_s.
\end{equation}
\end{proposition}
\begin{proof}
At a diagonal matrix with distinct eigenvalues, the first derivatives of a conjugation-invariant function in off-diagonal directions vanish, and
\[
 \frac{\partial^2F}{\partial T_{ij}\partial T_{ji}}
 =\frac{\partial_iF-\partial_jF}{t_i-t_j}\quad(i\ne j).
\]
This follows, for example, by perturbing the $2\times2$ block in rows and columns $i,j$: its eigenvalues change by $\pm xy/(t_i-t_j)$ to second order when the off-diagonal entries are $x,y$. Applying Lemma~\ref{lem:matrix-casimir}, the coefficient of $\partial_i^2$ in $C_{12}/2$ is $t_i(t_i-1)$, and its first-order coefficient is
\[
 Nt_i-s+\sum_{j\ne i}\frac{2t_it_j-t_i-t_j}{t_i-t_j}.
\]
The identity
\[
 \frac{2t_it_j-t_i-t_j}{t_i-t_j}
 =\frac{2t_i(t_i-1)}{t_i-t_j}-(2t_i-1)
\]
gives \eqref{eq:radial}. The apparent poles cancel on symmetric polynomials, since every paired numerator contains $\partial_iF-\partial_jF$.

Interchanging $X_1$ and $X_2$ takes the new cross-ratio to a matrix conjugate to $T^{-1}$, and the determinant prefactor is multiplied by $(\det T)^r$. Thus the interchange acts by $P_s$. It exchanges the pairs $(1,2)$ and $(3,2)$, proving \eqref{eq:second-radial}. The operator $P_s$ preserves the individual degree bounds and satisfies $P_s^2=1$.
\end{proof}

\subsection{The exterior-power comparison}\label{sec:exterior}
Set
\begin{equation}\label{eq:shift}
 d=r+s-1,\qquad N'=N-2s+2,\qquad V_d=\C[t]_{\leq d},
\end{equation}
and define operators on $V_d$ by
\begin{equation}\label{eq:one-variable}
 L=t(t-1)\partial_t^2+(N't-1)\partial_t,\qquad
 (P_df)(t)=t^df(t^{-1}).
\end{equation}
For $X\in\End(V_d)$, let $\dd\Gamma(X)$ be its additive exterior power:
\[
 \dd\Gamma(X)(v_1\wedge\cdots\wedge v_s)
 =\sum_{i=1}^s v_1\wedge\cdots\wedge Xv_i\wedge\cdots\wedge v_s.
\]
In contrast, $C_s(G)=\bigwedge^sG$ denotes the multiplicative exterior power of a linear map $G$; its matrix entries are the $s\times s$ minors of $G$.

Put
\begin{equation}\label{eq:epsilon-kappa}
 \epsilon_a=a(a+N'-1),\qquad
 \kappa_s=\sum_{a=0}^{s-1}\epsilon_a,\qquad
 \Delta(t)=\prod_{i<j}(t_j-t_i).
\end{equation}
Multiplication by $\Delta$ identifies the space in \eqref{eq:polynomial-model} with alternating polynomials of degree at most $d$ in each variable, and hence with $\bigwedge^sV_d$. Indeed, every alternating polynomial is divisible by $\Delta$; the quotient is symmetric and has degree at most $d-s+1=r$ in each variable. Denote this isomorphism by $U$.

\begin{theorem}\label{thm:two-residues}
The isomorphism $U$ satisfies
\begin{align}
 U(C_{12}/2)U^{-1}&=\dd\Gamma(L)-\kappa_s I,\label{eq:exterior-first}\\
 U(C_{23}/2)U^{-1}&=\dd\Gamma(P_dLP_d)-\kappa_s I.\label{eq:exterior-second}
\end{align}
The vacuum $F=1$ goes to the filled exterior vacuum
\begin{equation}\label{eq:filled}
 p_0\wedge p_1\wedge\cdots\wedge p_{s-1},
\end{equation}
where $p_a$ is the monic eigenpolynomial of $L$ of degree $a$.
\end{theorem}
\begin{proof}
The monomial action is
\begin{equation}\label{eq:L-monomial}
 L(t^a)=\epsilon_a t^a-a^2t^{a-1}.
\end{equation}
Since $N'\geq2$, the numbers $\epsilon_0,\ldots,\epsilon_d$ are distinct. Thus the monic eigenpolynomials $p_a$ exist and form a basis. Their alternant equals the Vandermonde:
\[
 \det(p_{i-1}(t_j))_{i,j=1}^s=\Delta(t).
\]
It follows that $\sum_iL_i\Delta=\kappa_s\Delta$. The product rule now gives
\[
 \Delta^{-1}\Bigl(\sum_iL_i\Bigr)\Delta
 =D_s+\kappa_s I,
\]
since the additional first-order coefficient is
\[
 2t_i(t_i-1)\partial_i\log\Delta
 =2t_i(t_i-1)\sum_{j\ne i}(t_i-t_j)^{-1}.
\]
This proves \eqref{eq:exterior-first} and \eqref{eq:filled}.

The Vandermonde inversion identity is
\[
 \Delta(t_1^{-1},\ldots,t_s^{-1})
 =(-1)^{\binom{s}{2}}\Delta(t)\prod_i t_i^{-(s-1)}.
\]
Consequently,
\begin{equation}\label{eq:P-intertwiner}
 UP_s=(-1)^{\binom{s}{2}}C_s(P_d)U.
\end{equation}
Conjugate \eqref{eq:exterior-first} using \eqref{eq:P-intertwiner} and Proposition~\ref{prop:radial}. The sign occurs twice and cancels, proving \eqref{eq:exterior-second}.
\end{proof}

For $\lambda\subseteq(r^s)$, let
\[
 I(\lambda)=\{\lambda_i+s-i:1\leq i\leq s\},\qquad I_0=\{0,\ldots,s-1\}.
\]
Index sets in minors are always listed increasingly. Reversing both rows and columns in \eqref{eq:B} does not change its determinant. The identity
\begin{equation}\label{eq:content-eigen}
 \sum_{a\in I(\lambda)}\epsilon_a-\kappa_s
 =N|\lambda|+2\ct(\lambda)
\end{equation}
follows by substituting $a_i=\lambda_i+s-i$ and using
$2\ct(\lambda)=\sum_i\lambda_i(\lambda_i-2i+1)$.
Every nonempty $\lambda$ has a strictly positive value in \eqref{eq:content-eigen}: the contribution of any box is $N+2(j-i)\geq2$. Thus zero is a simple eigenvalue of $C_{12}/2$, with vacuum vector $1$. Its exterior index is $I_0$.

\section{Monodromy and the reduced closure}\label{sec:closure}
\subsection{From KZ systems to pure braids}
Use the trace form on $\gl_N$, normalized so that its restriction to $\slie_N$ gives long roots squared length two. The scalar quadratic Casimir on $R=(r^s)$ is
\[
 c_R=rs(N+r-s).
\]
On the four-point invariant space, the KZ equation is
\begin{equation}\label{eq:KZ}
 \frac{\dd G}{\dd z}=h\left(\frac{\Omega_{12}}z+
 \frac{\Omega_{23}}{z-1}\right)G,
 \qquad \Omega_{ij}=C_{ij}/2-c_RI.
\end{equation}
For the one-row color $[d]$ at rank $N'$, its two residues are
\[
 L-c_dI,\qquad P_dLP_d-c_dI,\qquad c_d=d(d+N'-1).
\]
By Theorem~\ref{thm:two-residues}, the rectangular residues are their additive exterior powers plus the same scalar
\begin{equation}\label{eq:scalar-residue}
 \sigma=s c_d-\kappa_s-c_R.
\end{equation}
Choose a common basepoint $z_*\in(0,1)$ and based counterclockwise loops around $0$ and $1$, with paths to the punctures along the real interval. Use these choices for both systems. If $G(z)$ is a fundamental solution of the one-row system, $C_s(G(z))$ solves the system with the additive exterior residues. Adding scalar residues $aI,bI$ multiplies a fundamental solution by $z^{ha}(1-z)^{hb}$. The two based monodromies are therefore multiplied by $e^{2\pi iha}$ and $e^{2\pi ihb}$, respectively. Both identities use the same identification of the solution spaces.

The quasi-Hopf form of the Drinfeld--Kohno theorem identifies the KZ and formal Drinfeld--Jimbo braided categories by an algebra isomorphism and a tensor twist \cite{Drinfeld,Kohno}; see also \cite[Theorem 7.3]{Brochier}. In particular, it identifies the pure-braid actions on the alternating tensor product and on its invariant space by a simultaneous conjugacy. We apply the theorem to $\slie_N$. Restoring the one-dimensional center of $\gl_N$ adds the scalar residues determined by the central characters of $R$ and $R^*$.

Set $q=e^{\pi ih}$. Divide each adjacent-pair monodromy by its vacuum eigenvalue; write $M_0,M_1$ for the resulting operators. The eigenvalue attached to a pair Casimir eigenvalue $c$ is then $q^c$. For the one-row system denote the corresponding operators by $M'_0,M'_1$. Define
\begin{equation}\label{eq:one-eigenvalues}
 A'=q^{N'}=Aq^{2(1-s)},\qquad
 \theta_a=(A')^{2a}q^{2a(a-1)},\qquad
 \eta_s=\left(\prod_{a=0}^{s-1}\theta_a\right)^{-1}=q^{-2\kappa_s}.
\end{equation}
The scalar \eqref{eq:scalar-residue} and the external Casimir scalars cancel after vacuum normalization. The filled vacuum contributes $\prod_{a<s}\theta_a$, which must still be divided out.

\begin{proposition}\label{prop:quantum-pair}
Over a formal Laurent-series field containing the channel normalizations, the pair $(M_0,M_1)$ for $(r^s)$ at rank $N$ is simultaneously conjugate to
\begin{equation}\label{eq:quantum-pair}
 \left(\eta_s C_s(M'_0),\ \eta_s C_s(M'_1)\right)
\end{equation}
for $[d]$ at rank $N'$. Both vacuum spectral projectors are transported by this simultaneous conjugacy.
\end{proposition}
\begin{proof}
Take exterior powers of fundamental solutions of the one-row KZ equation and use \eqref{eq:scalar-residue}. Their monodromies are the multiplicative exterior powers with the scalar factors just described. Divide by the rectangular vacuum eigenvalues to obtain \eqref{eq:quantum-pair}. Apply the Drinfeld--Kohno equivalence to each of the two systems. Finally, the vacuum eigenvalue is simple by \eqref{eq:content-eigen}; its spectral projector is a rational function of its operator, and hence is preserved by conjugacy.
\end{proof}
We take spectral projectors over $\C((h))$, adjoining the channel normalizations when necessary. The operators are diagonalizable there, and differences between distinct formal eigenvalues are invertible. Thus the projectors are defined over this field even though all twist eigenvalues specialize to one at $h=0$.

\subsection{Vacuum projectors and the reduced plat}
Let $M_0,M_1$ be invertible, diagonalizable operators, each with simple eigenvalue one, and let $P_0,P_1$ be their corresponding spectral projectors. For $m,n\in\Z$, if $\Tr(P_0P_1)\ne0$, put
\begin{equation}\label{eq:trace-closure}
 \mathscr J_{m,n}(M_0,M_1)
 =\frac{\Tr(P_0M_1^mM_0^nP_1)}{\Tr(P_0P_1)}.
\end{equation}
The trace is the ordinary trace of a finite-dimensional vector-space operator.

\begin{lemma}\label{lem:intrinsic-closure}
The scalar \eqref{eq:trace-closure} is invariant under arbitrary simultaneous conjugation. If
\[
 M_0=D,\qquad M_1=FDF^{-1},\qquad D_{00}=1,
\]
and the eigenvalue one occurs only at index zero, then
\begin{equation}\label{eq:matrix-closure}
 \mathscr J_{m,n}(D,FDF^{-1})
 =\frac{(FD^mF^{-1}D^nF)_{00}}{F_{00}}.
\end{equation}

\end{lemma}
\begin{proof}
A projector is the polynomial
\[
 P_i=\prod_{\tau\in\Spec(M_i)\setminus\{1\}}
 \frac{M_i-\tau I}{1-\tau},
\]
where each distinct eigenvalue occurs once. Thus all four operators in \eqref{eq:trace-closure} transform by the same conjugation; invariance follows from cyclicity of trace. In the displayed basis,
\[
 P_0=e_0e_0^T,\qquad P_1=Fe_0e_0^TF^{-1}.
\]
The numerator in \eqref{eq:trace-closure} is
$(FD^mF^{-1}D^nF)_{00}(F^{-1})_{00}$, and the denominator is
$F_{00}(F^{-1})_{00}$. Cancelling the nonzero overlap gives \eqref{eq:matrix-closure}.
\end{proof}
\begin{figure}[tb]
\centering
\begin{tikzpicture}[x=0.9cm,y=0.65cm,line width=0.65pt,>=Stealth]
\foreach \x in {0,1,2,3}{\draw (\x,0)--(\x,5);}
\draw (0,5) .. controls (0,6) and (1,6) .. (1,5);
\draw (2,5) .. controls (2,6) and (3,6) .. (3,5);
\draw (1,0) .. controls (1,-0.9) and (2,-0.9) .. (2,0);
\draw (0,0) .. controls (0,-2.2) and (3,-2.2) .. (3,0);
\draw[fill=white] (0.75,3.2) rectangle (2.25,4.1);
\node at (1.5,3.65) {$M_1^m$};
\draw[fill=white] (-0.25,1.25) rectangle (1.25,2.15);
\node at (0.5,1.7) {$M_0^n$};
\draw[->] (0,4.7)--(0,4.35);\draw[->] (1,4.35)--(1,4.7);
\draw[->] (2,4.7)--(2,4.35);\draw[->] (3,4.35)--(3,4.7);
\node[left] at (-0.25,1.7) {$2n$ crossings};
\node[right] at (3.2,3.65) {$2m$ crossings};
\node at (0.5,6.05) {$(12)$};\node at (2.5,6.05) {$(34)$};
\node[right] at (3,-0.8) {$(14)(23)$};
\end{tikzpicture}
\caption{The alternating four-plat. The upper caps pair $(12)(34)$ and the lower cups pair $(14)(23)$. The boxes contain even adjacent braids, acting from bottom to top. A positive unit of either box has oriented writhe $-2$. Multiplying its RT operator by the external ribbon correction gives the vacuum-normalized operator indicated in the figure.}\label{fig:plat}
\end{figure}
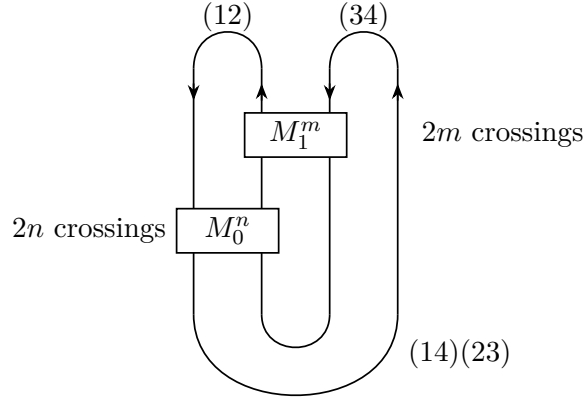

The unreduced quantum dimension is
\begin{equation}\label{eq:quantum-dimension}
 d_\lambda(A,q)=\prod_{\square\in\lambda}
 \frac{D_{c(\square)}}{\br{h(\square)}}.
\end{equation}
At $A=q^N$ and $\ell(\lambda)\leq N$, this is the quantum dimension of
$\mathbb S_\lambda\C^N$.

\begin{proposition}\label{prop:RT-closure}
For the reduced zero-framed double-twist invariant,
\begin{equation}\label{eq:rt-trace}
 H_R(K_{m,n};q^N,q)=\mathscr J_{m,n}(M_0,M_1).
\end{equation}
Moreover, $\Tr(P_0P_1)=(\dim_q R)^{-2}\ne0$.
\end{proposition}
\begin{proof}
Work in the pivotal ribbon category of quantum-group representations \cite{RT} and set $d_R=\dim_qR$. Let $u$ be the raw cup state pairing $(12)(34)$, and let $v$ be the raw cup state pairing $(14)(23)$. Let $u^*,v^*$ be their planar cap covectors. The compositions are evaluations of disjoint circles:
\[
 u^*u=v^*v=d_R^2,\qquad u^*v=v^*u=d_R.
\]
The orientations pair each evaluation with the corresponding dual module. Set $u_0=u/d_R$, $v_0=v/d_R$, $\ell_0=u^*/d_R$, and $\ell_1=v^*/d_R$. The two rank-one operators
\[
 P_0=u_0\ell_0,\qquad P_1=v_0\ell_1
\]
are precisely the vacuum spectral projectors. The cap map vanishes on a nontrivial adjacent channel, and the trivial channel is the unique channel with relative twist eigenvalue one. Therefore
\[
 \Tr(P_0P_1)=(\ell_0v_0)(\ell_1u_0)=d_R^{-2}.
\]

For $W=M_1^mM_0^n$, the numerator of \eqref{eq:trace-closure} is
$(\ell_0Wv_0)(\ell_1u_0)$. Its ratio with the denominator is
\[
 \frac{\ell_0Wv_0}{\ell_0v_0}=\frac{u^*Wv}{d_R}.
\]
The last expression is the reduced evaluation of the plat in Figure~\ref{fig:plat}. Opening the outer lower cup exhibits the usual two twist regions of the double-twist diagram.

It remains to identify its framing. Ribbon balancing gives the eigenvalue $q^{c_Q-2c_R}$ for a positive pure braid on $R\otimes R^*$ in channel $Q$. Its oriented writhe is $-2$, so the zero-framing correction is $q^{2c_R}$. The corrected eigenvalue is $q^{c_Q}$, exactly the vacuum normalization used to define $M_i$. Applying the inverse correction handles a negative twist. The operator $W$ therefore includes the zero-framing correction in both boxes, and division by $d_R$ gives the reduced invariant. This proves \eqref{eq:rt-trace} for all integral $m,n$.
\end{proof}

We now evaluate the closure using the one-row color $[d]$ at rank $N'$. Let $S$ be its symmetric orthogonal exclusive Racah matrix. This is the kind-I matrix $\bar S$ in \cite[Section 2, equations (1), (4), and (5)]{MMS}, with the overall sign chosen so that $S_{00}=(\dim_q[d])^{-1}$. Its normalization is
\begin{equation}\label{eq:S-properties}
 S^T=S,\qquad S^2=I,\qquad
 S_{0a}=\frac{\sqrt{\delta_a}}{\gamma},\qquad
 \gamma=\dim_q[d],\qquad \delta_0=1.
\end{equation}
Here $\delta_a$ is the quantum dimension of the $a$th summand of
$[d]\otimes[d]^*$. Closing a normalized trivalent vertex gives the vacuum row in \eqref{eq:S-properties}. Reflection gives $S^T=S$, and composing the two changes of fusion basis gives $S^2=I$. We choose the square roots with positive constant term at $q=1$.

The one-row decomposition can also be seen by the contraction filtration of
$\Sym^a\C^{N'}\otimes\Sym^a(\C^{N'})^*$: its successive harmonic components have highest weights
$a\omega_1+a\omega_{N'-1}$. These components occur once and deform to the generic quantum category. Consequently their dimensions are
\[
 \delta_a=\chi_a^2-\chi_{a-1}^2,\qquad
 \chi_a=\qb{N'+a-1}{a},\qquad \chi_{-1}=0.
\]
In particular, each $\delta_a$ is nonzero, since its specialization at $q=1$ is positive for $N'\geq2$.

In this basis,
\[
 M'_0=\theta:=\diag(\theta_0,\ldots,\theta_d),\qquad
 M'_1=S\theta S.
\]
Put $\widehat S=C_s(S)$ and $\widehat D=\eta_sC_s(\theta)$. The exterior vacuum is indexed by $I_0$. Proposition~\ref{prop:quantum-pair} and Lemma~\ref{lem:intrinsic-closure} give the following formula.

\begin{theorem}\label{thm:rt-determinant}
For $N\geq2s$,
\begin{equation}\label{eq:rt-determinant}
 H_{(r^s)}(K_{m,n};q^N,q)
 =\eta_s^{m+n}
 \frac{\det\bigl((S\theta^mS\theta^nS)_{I_0,I_0}\bigr)}
 {\det(S_{I_0,I_0})}.
\end{equation}
The denominator is nonzero.
\end{theorem}
\begin{proof}
By Proposition~\ref{prop:RT-closure}, the invariant is the trace ratio of the pair \eqref{eq:quantum-pair}. In the exterior basis this pair is
$(\widehat D,\widehat S\widehat D\widehat S)$, with $\widehat S^2=I$. Apply \eqref{eq:matrix-closure} at index $I_0$ and use the multiplicativity of compound matrices. The product of the two powers of $\eta_s$ gives the displayed scalar. Also,
\[
 \Tr(P_0P_1)=(\widehat S_{I_0,I_0})^2
 =\det(S_{I_0,I_0})^2=(\dim_q(r^s))^{-2}.
\]
This proves the nonvanishing assertion.
\end{proof}

\section{Gaussian determinants and the double-twist formula}\label{sec:double-twist-proof}
The partition sum has an exterior-power description parallel to that of
the four-point invariant space. We first derive its determinant and then
compare it with Theorem~\ref{thm:rt-determinant}.

\subsection{The Gaussian matrix and its exterior powers}
Return to independent $A,q$ and set
\begin{equation}\label{eq:particle-data}
 d=r+s-1,\qquad A'=Aq^{2(1-s)},\qquad y=q^{1-s},\qquad
 \theta_a=(A')^{2a}q^{2a(a-1)}.
\end{equation}
Define a $(d+1)\times(d+1)$ lower-triangular matrix and a diagonal shape matrix by
\begin{equation}\label{eq:bY}
 b_{ac}=\theta_c\Gb ac,\qquad
 Y_{\lambda\lambda}=y^{|\lambda|},\qquad
 \eta_s=\left(\prod_{a=0}^{s-1}\theta_a\right)^{-1}.
\end{equation}

\begin{proposition}\label{prop:compound}
On $\cP_{r,s}$,
\begin{equation}\label{eq:compound}
 B^{(s)}=\eta_sY C_s(b)Y^{-1}.
\end{equation}
Consequently, for every $m\in\Z$,
\begin{equation}\label{eq:powers}
 ((B^{(s)})^m)_{\lambda\mu}
 =\eta_s^m y^{|\lambda|-|\mu|}
 \det\bigl((b^m)_{I(\lambda),I(\mu)}\bigr).
\end{equation}
Furthermore, \eqref{eq:B-first} holds.
\end{proposition}
\begin{proof}
The particle statistics give
\begin{equation}\label{eq:particle-theta}
 \Theta_\mu=\eta_s\prod_{a\in I(\mu)}\theta_a.
\end{equation}
Indeed, $\sum_{a\in I(\mu)}a-\sum_{a\in I_0}a=|\mu|$, and
\[
 \sum_{a\in I(\mu)}a(a-1)-\sum_{a\in I_0}a(a-1)
 =2\ct(\mu)+2(s-1)|\mu|.
\]
Substituting $A'=Aq^{-2(s-1)}$ proves \eqref{eq:particle-theta}. Factoring the $\theta$'s from the columns of the minor gives \eqref{eq:compound}. Multiplicativity of exterior powers proves \eqref{eq:powers}, including inverse powers.

For the first column, regard $\Gb aj$, $0\leq j<s$, as a polynomial in $q^{2a}$ of degree $j$. Its leading coefficient is independent of $a$. Divide its alternant at the particle set $I(\lambda)$ by its alternant at $I_0$. The latter determinant is one. The quotient of the two Vandermondes is
\[
 \prod_{i<j}\frac{q^{2a_j}-q^{2a_i}}{q^{2(j-1)}-q^{2(i-1)}}
 =q^{(s-1)|\lambda|}\prod_{i<j}\frac{[a_j-a_i]_q}{[j-i]_q}
 =q^{(s-1)|\lambda|}D_\lambda^{(s)},
\]
where $a_1<\cdots<a_s$ list $I(\lambda)$. The prefactor in \eqref{eq:B} cancels this monomial, giving \eqref{eq:B-first}.
\end{proof}

\begin{proposition}\label{prop:height}
For a fixed partition $\lambda$ and every $m\in\Z$, the normalized ratio
\[
 \frac{((B^{(s)})^m)_{\lambda\vac}}{B^{(s)}_{\lambda\vac}}
\]
is independent of the ambient height $s\geq\max\{1,\ell(\lambda)\}$ and of any containing width.
\end{proposition}
\begin{proof}
The width statement follows from triangularity: all intermediate partitions in an entry lie in the finite interval $[\vac,\lambda]$. This holds for inverse matrices as well, by the triangular inverse recursion.

For the height step, suppose $\lambda_s=\mu_s=0$ and let
$a_i=\lambda_i+s-1-i$ for $1\leq i<s$. Define
\[
 g_{s-1}(\lambda)=\prod_{i=1}^{s-1}q^{-a_i}(1-q^{2(a_i+1)}).
\]
In the height-$s$ Gaussian minor, the row indexed by $0$ has a single nonzero entry, in the column indexed by $0$. Expanding along this row removes both indices. For each remaining entry use
\[
 \Gb{a+1}{c+1}=\frac{1-q^{2(a+1)}}{1-q^{2(c+1)}}\Gb ac.
\]
Comparison of the prefactors in \eqref{eq:B} gives
\begin{equation}\label{eq:height}
 B^{(s)}_{(\lambda,0),(\mu,0)}
 =\frac{g_{s-1}(\lambda)}{g_{s-1}(\mu)}B^{(s-1)}_{\lambda\mu}.
\end{equation}
In a finite rectangular truncation, the states of exact height $s$ span an invariant subspace, and \eqref{eq:height} describes the induced operator on the lower-height quotient. The same diagonal conjugacy holds for every integral power. Its first-column gauge factor cancels upon division by the first column of $B$, proving the assertion by induction.
\end{proof}

\begin{proposition}\label{prop:mirror}
Let an overline mean substitution $(A,q)\mapsto(A^{-1},q^{-1})$, and let $T_0=\diag(t_\lambda)$. Then
\begin{equation}\label{eq:mirror}
 \overline{B^{(s)}}=T_0(B^{(s)})^{-1}T_0^{-1}.
\end{equation}
For every $m\ne0$,
\begin{equation}\label{eq:all-twists}
 \Tw_\lambda^{(m)}=t_\lambda
 \frac{((B^{(s)})^m)_{\lambda\vac}}{D_\lambda^{(s)}}.
\end{equation}
\end{proposition}
\begin{proof}
Let $P_{ac}=\Gb ac$. Gaussian inversion gives
\begin{equation}\label{eq:pascal-inverse}
 (P^{-1})_{ac}=(-1)^{a-c}q^{(a-c)(a-c-1)}\Gb ac.
\end{equation}
To verify the inverse, multiply by $P$ and use
\[
 \sum_{j=0}^k(-1)^jq^{j(j-1)}\Gb kj
 =\prod_{j=0}^{k-1}(1-q^{2j}).
\]
This equals zero for $k>0$ and one for $k=0$.

Set $t'_a=(-A'^2)^a q^{a(a-1)}$ and $T'=\diag(t'_a)$. Since $b=P\theta$, \eqref{eq:pascal-inverse} implies
\[
 \overline b=T'b^{-1}(T')^{-1}.
\]
Entrywise this is the identity
\[
 \frac{t'_a}{t'_c}\theta_a^{-1}(-1)^{a-c}q^{(a-c)(a-c-1)}
 =\theta_c^{-1}q^{-2c(a-c)}.
\]
The remaining factor is $\Gb ac$, while
$\overline{\Gb ac}=q^{-2c(a-c)}\Gb ac$.
The particle product is
\begin{equation}\label{eq:t-particle}
 \frac{\prod_{a\in I(\lambda)}t'_a}{\prod_{a\in I_0}t'_a}
 =t_\lambda y^{2|\lambda|}.
\end{equation}
Using \eqref{eq:compound}, $\overline Y=Y^{-1}$, and $\overline\eta_s=\eta_s^{-1}$, the factors $y^{2|\lambda|}$ in \eqref{eq:t-particle} give precisely \eqref{eq:mirror}.

For $m=-k<0$, the definition in \eqref{eq:kntz-twists} and the $q\leftrightarrow q^{-1}$ invariance of $D_\lambda^{(s)}$ give
\[
 \Tw_\lambda^{(-k)}=\frac{(\overline B^{\,k})_{\lambda\vac}}{D_\lambda^{(s)}}
 =t_\lambda\frac{(B^{-k})_{\lambda\vac}}{D_\lambda^{(s)}}.
\]
This proves \eqref{eq:all-twists}; its positive case is the definition. Taking $k=1$ also recovers $\Tw_\lambda^{(-1)}=1$.
\end{proof}

Write $\cK_{r,s}^{m,n}$ for the sum of \eqref{eq:main}. Substitution of \eqref{eq:all-twists} into the right-hand side of \eqref{eq:main} gives
\begin{equation}\label{eq:shape-bilinear}
 \cK_{r,s}^{m,n}=\sum_{\lambda\subseteq(r^s)}
 W_\lambda(B^m)_{\lambda\vac}(B^n)_{\lambda\vac},
\end{equation}
where
\begin{equation}\label{eq:shape-weight}
 W_\lambda=t_\lambda\frac{D_{\lambda^T}^{(r)}(q^{-1})}{D_\lambda^{(s)}(q)}
 \prod_{\square\in\lambda}D_{r+c(\square)}D_{c(\square)-s}.
\end{equation}
Define one-particle weights by
\begin{equation}\label{eq:one-weight}
 w_a=(-A'^2)^a q^{a(a-1)}\qb da
 \prod_{k=1}^a\{A'q^{d+k-1}\}\{A'q^{k-2}\},\qquad
 W=\diag(w_0,\ldots,w_d).
\end{equation}
All these weights are nonzero in $\Q(A,q)$ and at the stable-rank specializations used in the proof.

\begin{lemma}\label{lem:weight}
For $\lambda\subseteq(r^s)$,
\begin{equation}\label{eq:weight-factor}
 W_\lambda y^{2|\lambda|}
 =\frac{\prod_{a\in I(\lambda)}w_a}{\prod_{a\in I_0}w_a}.
\end{equation}
\end{lemma}
\begin{proof}
The hook factors cancel between the two quantum dimensions, leaving
\[
 \frac{D_{\lambda^T}^{(r)}(q^{-1})}{D_\lambda^{(s)}(q)}
 =\prod_{(i,k)\in\lambda}\frac{[r+i-k]_q}{[s+k-i]_q}.
\]
For row $i$, put $j_0=s-i$. The successive particle values in that row are $a=j_0+1,\ldots,j_0+\lambda_i$. The row's dimension contribution is
\[
 \prod_{a=j_0+1}^{j_0+\lambda_i}\frac{[d+1-a]_q}{[a]_q}
 =\frac{\qb d{j_0+\lambda_i}}{\qb d{j_0}}.
\]
For the box $(i,k)$, with $a=s-i+k$, the two differential arguments satisfy
\[
 A'q^{d+a-1}=Aq^{r+k-i},\qquad
 A'q^{a-2}=Aq^{k-i-s}.
\]
Thus their products are exactly the corresponding interval factors in
$w_{j_0+\lambda_i}/w_{j_0}$. Finally, \eqref{eq:t-particle} supplies the monomial and sign. Multiplication over the rows proves \eqref{eq:weight-factor}.
\end{proof}

\begin{theorem}\label{thm:shape-determinant}
For every pair of nonzero integers $m,n$,
\begin{equation}\label{eq:K-det}
 \cK_{r,s}^{m,n}(A,q)=\eta_s^{m+n}
 \frac{\det\bigl(((b^m)^TWb^n)_{I_0,I_0}\bigr)}
 {\prod_{a=0}^{s-1}w_a}.
\end{equation}
\end{theorem}
\begin{proof}
Insert \eqref{eq:powers} into \eqref{eq:shape-bilinear} and then apply \eqref{eq:weight-factor}. The result is
\[
 \frac{\eta_s^{m+n}}{\prod_{a\in I_0}w_a}
 \sum_{|I|=s}\left(\prod_{a\in I}w_a\right)
 \det((b^m)_{I,I_0})\det((b^n)_{I,I_0}).
\]
Cauchy--Binet identifies this sum with the initial principal minor of $(b^m)^TWb^n$.
\end{proof}

\subsection{The one-row identity and initial principal minors}\label{sec:moments}
Specialize $A'=q^{N'}$, where $N'\geq2$, and let $S$ be the one-row Racah matrix normalized by \eqref{eq:S-properties}. We first translate the skein formula of \cite[Theorem 1.2]{CLZ} into a moment identity for $b$.

\begin{proposition}\label{prop:seed-moment}
For all positive integers $m,n$,
\begin{equation}\label{eq:seed-moments}
 e_0^T(b^m)^TWb^ne_0
 =\gamma(S\theta^mS\theta^nS)_{00},\qquad
 \gamma=\dim_q[d]=1/S_{00}.
\end{equation}
\end{proposition}
\begin{proof}
The right-hand side is the RT one-row double-twist invariant, by Proposition~\ref{prop:RT-closure} and \eqref{eq:matrix-closure}. We obtain the left-hand side from the Chen--Liu--Zhu skein formula as follows.

Write $a=A'$ temporarily. For $p>0$, their coefficient \cite[equation (3.12)]{CLZ} is
\begin{equation}\label{eq:CLZ-coeff}
 \widetilde C_{k,k}^{(p)}
 =(-1)^k a^{-k}q^{-2k^2}
 \sum_{\ell_1+\cdots+\ell_p=k}
 a^{-2\sum_{i=1}^{p-1}(p-i)\ell_i}
 q^{\varphi(\ell)}\frac{[k]_q!}{\prod_{i=1}^p[\ell_i]_q!},
\end{equation}
where all $\ell_i$ are nonnegative, $s_i=\ell_1+\cdots+\ell_i$, and
\begin{equation}\label{eq:CLZ-phi}
 \varphi(\ell)=2\sum_{i=1}^{p-1}s_i(s_i-2k+1)
 +\frac12\sum_{i=1}^p\ell_i(\ell_i+3)
 +2\sum_{i=1}^{p-1}s_i\ell_{i+1}.
\end{equation}
In a path contributing to $(b^p)_{k0}$, put $j_0=k$, $j_i=k-s_i$, and $j_p=0$. Multiplication of its Gaussian binomials gives
\[
 \prod_{i=1}^p\Gb{j_{i-1}}{j_i}
 =q^{\sum_{u<v}\ell_u\ell_v}
 \frac{[k]_q!}{\prod_i[\ell_i]_q!}.
\]
The $a$- and $q$-exponents in \eqref{eq:CLZ-coeff} agree with this path by
\begin{align}
 2pk-2k-2\sum_{i=1}^{p-1}(p-i)\ell_i
 &=2\sum_{i=1}^{p-1}j_i,\label{eq:CLZ-A}\\
 2pk(k-1)-2k^2-\tfrac12k(k-1)+\varphi(\ell)
 &=2\sum_{i=1}^{p-1}j_i(j_i-1)+\sum_{u<v}\ell_u\ell_v.\label{eq:CLZ-q}
\end{align}
The first identity uses $\sum(p-i)\ell_i=\sum s_i$. For the second, substitute $j_i=k-s_i$ and use
$\sum s_i\ell_{i+1}=\sum_{u<v}\ell_u\ell_v=(k^2-\sum\ell_i^2)/2$.
Consequently,
\begin{equation}\label{eq:CLZ-conversion}
 (a^kq^{k(k-1)})^{2p}\widetilde C_{k,k}^{(p)}
 =(-1)^ka^kq^{k(k-1)/2}(b^p)_{k0}.
\end{equation}
This identity also holds at $k=0$.

Theorem 1.2 of \cite{CLZ} expresses the invariant with external color $[d]$ as the sum over $0\leq k\leq d$ of
\[
 (-1)^k(a^kq^{k(k-1)})^{2(m+n)}
 \widetilde C_{k,k}^{(m)}\widetilde C_{k,k}^{(n)}
 \qb dk\prod_{j=1}^k\{aq^{d+j-1}\}\{aq^{j-2}\}.
\]
Substituting \eqref{eq:CLZ-conversion} twice leaves
$w_k(b^m)_{k0}(b^n)_{k0}$. Summation proves \eqref{eq:seed-moments}.
\end{proof}

The following factorization converts the moment identity into a comparison of initial principal minors.

\begin{lemma}\label{lem:moment-lift}
Let $d\geq0$ and let $b$ be an invertible lower-triangular $(d+1)\times(d+1)$ matrix over a field of characteristic zero. Suppose its diagonal entries $\theta_0,\ldots,\theta_d$ are pairwise distinct, and set $\theta=\diag(\theta_0,\ldots,\theta_d)$. Let $W$ be an invertible diagonal matrix and let $S$ satisfy $S^T=S$, $S^2=I$, and $S_{0a}\ne0$ for every $a$. Assume that, for some $\gamma\ne0$,
\begin{equation}\label{eq:finite-moments}
 e_0^T(b^m)^TWb^ne_0
 =\gamma(S\theta^mS\theta^nS)_{00}
 \qquad(1\leq m,n\leq d+1).
\end{equation}
Then for every $1\leq s\leq d+1$ and every $m,n\in\Z$,
\begin{equation}\label{eq:minor-comparison}
 \frac{\det\bigl(((b^m)^TWb^n)_{I_0,I_0}\bigr)}
 {\prod_{a\in I_0}W_{aa}}
 =\frac{\det\bigl((S\theta^mS\theta^nS)_{I_0,I_0}\bigr)}
 {\det(S_{I_0,I_0})},\qquad I_0=\{0,\ldots,s-1\}.
\end{equation}
In particular, every denominator on the right is nonzero.
\end{lemma}
\begin{proof}
There is a unit lower-triangular eigenvector matrix $\mathsf E$ satisfying
$b\mathsf E=\mathsf E\theta$. Its columns are obtained by solving the rows from top to bottom; the differences $\theta_a-\theta_c$ in this recursion are nonzero. Set $v=\mathsf E^{-1}e_0$. Expanding the two sides of \eqref{eq:finite-moments} in the functions $\theta_a^m\theta_c^n$ gives
\begin{equation}\label{eq:Gram-moment}
 \diag(v)\mathsf E^TW\mathsf E\diag(v)
 =\gamma\diag(S_{0a})S\diag(S_{0a}).
\end{equation}
Indeed, the matrix $(\theta_a^m)_{1\leq m\leq d+1,\,0\leq a\leq d}$ factors as a Vandermonde matrix times a nonzero diagonal matrix. It is therefore invertible. Applying its inverse in both indices to the finite moment identities recovers every coefficient and gives \eqref{eq:Gram-moment}.

The right side of \eqref{eq:Gram-moment} is invertible. Since $\mathsf E$ and $W$ are invertible, this forces every $v_a$ to be nonzero. After adjoining $\sqrt\gamma$, set
\[
 C=\diag\left(\frac{\sqrt\gamma\,S_{0a}}{v_a}\right)\mathsf E^{-1}.
\]
Then $C$ is invertible and lower triangular, and \eqref{eq:Gram-moment} gives
\begin{equation}\label{eq:triangular-congruence}
 Cb=\theta C,\qquad W=C^TSC.
\end{equation}
Define $J=C^{-T}W$. Because $W$ is diagonal, $J$ is upper triangular. The second identity in \eqref{eq:triangular-congruence} gives $J=SC$, and hence $C=SJ$. For every integral $m,n$ it follows that
\begin{equation}\label{eq:full-congruence}
 (b^m)^TWb^n=J^T(S\theta^mS\theta^nS)J.
\end{equation}

Since $J$ is upper triangular, it satisfies $J_{a i}=0$ whenever $i<s\leq a$. Hence the $I_0$ principal block in \eqref{eq:full-congruence} is a product of the corresponding blocks, and
\[
 \det\bigl(((b^m)^TWb^n)_{I_0,I_0}\bigr)
 =\det(J_{I_0,I_0})^2
 \det\bigl((S\theta^mS\theta^nS)_{I_0,I_0}\bigr).
\]
At $m=n=0$, this becomes
\[
 \prod_{a\in I_0}W_{aa}
 =\det(J_{I_0,I_0})^2\det(S_{I_0,I_0}).
\]
All factors on the left and $\det J_{I_0,I_0}$ are nonzero. Division proves \eqref{eq:minor-comparison}. Both sides of the resulting identity belong to the original field, so the identity descends from the quadratic extension.
\end{proof}

\begin{proof}[Proof of Theorem~\ref{thm:double-twist}]
Fix $N\geq2s$. The one-particle eigenvalues
\[
 \theta_a=q^{2a(a+N'-1)}
\]
are nonzero and pairwise distinct. The weights in
\eqref{eq:one-weight} are also nonzero. Indeed, since $N'\geq2$, every
quantum differential occurring there has a positive exponent at
$A'=q^{N'}$. Proposition~\ref{prop:seed-moment} therefore verifies the
hypotheses of Lemma~\ref{lem:moment-lift}. Combining
Theorems~\ref{thm:rt-determinant} and \ref{thm:shape-determinant} with
\eqref{eq:minor-comparison} proves
\begin{equation}\label{eq:stable-equality}
 H_{(r^s)}(K_{m,n};q^N,q)=\cK_{r,s}^{m,n}(q^N,q)
\end{equation}
for all nonzero integral $m,n$.

The KZ argument first gives \eqref{eq:stable-equality} after $q=e^{\pi ih}$. Substitution defines an injection $\Q(q)\hookrightarrow\C((h))$: a nonzero polynomial in $q$ has finite order of vanishing at $q=1$, and its substitution consequently has finite order of vanishing at $h=0$. The same assertion holds for a nonzero rational function. Therefore \eqref{eq:stable-equality} holds in $\Q(q)$.

Finally fix $r,s,m,n$ and subtract the two sides of \eqref{eq:main}. Both sides belong to $\Q(A,q)$. Clear denominators and a possible power of $A$, obtaining a polynomial $P(A)\in\Q(q)[A]$. A fixed nonzero denominator can vanish at only finitely many of the distinct elements $q^N$ of $\Q(q)$. Equation~\eqref{eq:stable-equality} therefore gives $P(q^N)=0$ for infinitely many integers $N\geq2s$. A nonzero polynomial over a field cannot have infinitely many distinct roots. Hence $P=0$, proving the two-variable identity.
\end{proof}

\section{Coefficients and operations on knots}\label{sec:coefficients}
\subsection{The coefficients of double-twist knots}
Extend the transition coefficients of Theorem~\ref{thm:compression} by
$T_{\lambda,k}=0$ outside $\lambda_1\leq k\leq|\lambda|$, and set
$T_{\varnothing,0}=1$ and $T_{\varnothing,k}=0$ for $k>0$.
For double-twist knots, their relation to both rectangular expansions
is explicit.

\begin{theorem}\label{thm:coefficient-bridge}
For all $m,n\in\Z\setminus\{0\}$, $s\geq1$, and $k,l\geq0$,
\begin{align}
 C_{k,s}(K_{m,n})
 &=\sum_{\substack{\lambda\subseteq(k^s)\\k\leq|\lambda|}}
 G_\lambda^{m,n}T_{\lambda,k}(A,q)
       \Phi_{\lambda^T,s}(A,q^{-1}),
 \label{eq:bridge-row}\\
 E_{k,l}(K_{m,n})
 &=\sum_{\substack{\lambda\subseteq(k^l)\\\max(k,l)\leq|\lambda|}}
 G_\lambda^{m,n}T_{\lambda,k}(A,q)
       T_{\lambda^T,l}(A,q^{-1}).
 \label{eq:bridge-double}
\end{align}
Both finite sums belong to $\R$.
\end{theorem}
\begin{proof}
By Theorem~\ref{thm:double-twist} and \eqref{eq:two-sided-shape},
\[
 H_{r,s}(K_{m,n})=
 \sum_{\lambda\subseteq(r^s)}G_\lambda^{m,n}
       \Phi_{\lambda,r}(A,q)\Phi_{\lambda^T,s}(A,q^{-1}).
\]
Apply \eqref{eq:compression} to the first shape factor. The coefficient
of $Z^+_{r,k}$ is the sum in \eqref{eq:bridge-row}. Indeed,
$T_{\lambda,k}$ is supported on
$\lambda_1\leq k\leq|\lambda|$, so for $k\leq r$ every contributing
partition is already contained in $(k^s)$. The coefficient is
independent of $r$, and triangular uniqueness gives
\eqref{eq:bridge-row}.

Next apply \eqref{eq:compression} to both shape factors, with $q$
inverted in the second. The coefficient of $Z^+_{r,k}Z^-_{s,l}$ is
\[
 \sum_{\lambda\subseteq(r^s)}G_\lambda^{m,n}
 T_{\lambda,k}(A,q)T_{\lambda^T,l}(A,q^{-1}).
\]
Its nonzero terms have $\lambda_1\leq k$,
$\ell(\lambda)\leq l$, and $|\lambda|\geq\max(k,l)$.
For $k\leq r$ and $l\leq s$, this is exactly the index set in
\eqref{eq:bridge-double}; the coefficient is independent of the outer
rectangle. Successive triangular uniqueness proves
\eqref{eq:bridge-double} over $\Q(A,q)$.
Theorems~\ref{thm:row-main} and \ref{thm:double-main} show that the
identified coefficients lie in $\R$.
\end{proof}

The first identities are
\begin{equation}\label{eq:bridge-small}
 E_{1,1}=G_{(1)}^{m,n},\qquad
 E_{2,1}=D_0G_{(2)}^{m,n},\qquad
 E_{1,2}=D_0G_{(1,1)}^{m,n}.
\end{equation}
For the next square coefficient, the bottom transitions in
\eqref{eq:bottom-transition} give
\begin{equation}\label{eq:bridge-22}
 \begin{split}
 E_{2,2}(K_{m,n})={}&G_{(2)}^{m,n}+G_{(1,1)}^{m,n}
 +[2]_q^2D_1D_{-1}G_{(2,1)}^{m,n}\\
 &+D_1D_2D_{-1}D_{-2}G_{(2,2)}^{m,n}.
 \end{split}
\end{equation}
\subsection{Products and connected sums}\label{sec:operations}
\begin{proposition}\label{prop:product}
For every $i,j\geq0$, there are unique $L_{i,j}^k\in \R$, with $\max(i,j)\leq k\leq i+j$, such that
\begin{equation}\label{eq:product}
 Z^+_{r,i}Z^+_{r,j}
 =\sum_{k=\max(i,j)}^{i+j}L_{i,j}^kZ^+_{r,k}
 \qquad(r\geq0).
\end{equation}
In particular,
\begin{equation}\label{eq:multiply-one}
 Z^+_{r,i}Z^+_{r,1}
 =[i+1]_qZ^+_{r,i+1}+[i]_qD_iZ^+_{r,i}.
\end{equation}
\end{proposition}
\begin{proof}
The product $R_i(X)R_j(X)$ has degree $i+j$ and is divisible by $\prod_{a=0}^{\max(i,j)-1}(X-X_a)$. Its rational expansion in the basis $R_k$ therefore has the indicated support. At every $X_r$, the product is a Laurent polynomial of parity $(-1)^{i+j}$ and satisfies finite reflection. Theorem~\ref{thm:abstract} proves that all its coefficients are integral. Finally,
\[
 (X-X_0)R_i(X)=\br{i+1}R_{i+1}(X)+(X_i-X_0)R_i(X),
\]
and $X_i-X_0=\br{i}D_i$. Dividing by $\br{1}$ gives \eqref{eq:multiply-one}.
\end{proof}

\begin{corollary}\label{cor:connected-sum}
For zero-framed knots $K$ and $L$ and fixed height $s$,
\begin{equation}\label{eq:connected-sum}
 C_{k,s}(K\mathbin{\#}L)
 =\sum_{\substack{0\leq i,j\leq k\\k\leq i+j}}
 L_{i,j}^kC_{i,s}(K)C_{j,s}(L).
\end{equation}
In particular,
\[
 E_{1,1}(K\mathbin{\#}L)
 =E_{1,1}(K)+E_{1,1}(L)
  +D_1D_{-1}E_{1,1}(K)E_{1,1}(L).
\]
\end{corollary}
\begin{proof}
The reduced invariant is the scalar of the colored $(1,1)$-tangle. Composition of the two tangles gives $H_{r,s}(K\mathbin{\#}L)=H_{r,s}(K)H_{r,s}(L)$. Multiply the two row expansions, apply Proposition~\ref{prop:product}, and use uniqueness. Formula \eqref{eq:E11} gives the last identity.
\end{proof}

\subsection{The specialization at \texorpdfstring{$q=1$}{q=1}}
Put
\[
 P_K(A)=H_{1,1}(K;A,1),\qquad
 g_K(A)=E_{1,1}(K;A,1).
\]
By \eqref{eq:E11},
\begin{equation}\label{eq:special-g}
 P_K(A)=1+D_0^2g_K(A),\qquad g_K\in\Z[A^{\pm1}].
\end{equation}
The special-polynomial identity \cite[Theorem~1.1]{Zhu13} gives
\begin{equation}\label{eq:special-polynomial}
 H_{r,s}(K;A,1)=P_K(A)^{rs}.
\end{equation}

For $k,l\geq1$, write $[k]=\{1,\ldots,k\}$ and let $n_{k,l,e}$ be the number of subsets of $[k]\times[l]$ with $e$ pairs that meet every row and every column. Equivalently, it counts labeled simple bipartite graphs on the specified two vertex sets with $e$ edges and no isolated vertices. Thus
\begin{equation}\label{eq:graph-count}
 n_{k,l,e}=\sum_{i=0}^k\sum_{j=0}^l
 (-1)^{k+l-i-j}\binom{k}{i}\binom{l}{j}\binom{ij}{e}.
\end{equation}
It is zero unless $\max(k,l)\leq e\leq kl$.

\begin{theorem}\label{thm:special-double}
For every zero-framed knot $K$ and $k,l\geq1$,
\begin{equation}\label{eq:graph-specialization}
 E_{k,l}(K;A,1)
 =\sum_{e=\max(k,l)}^{kl}
 n_{k,l,e}\,D_0^{\,2e-k-l}g_K(A)^e.
\end{equation}
For every fixed height $s\geq1$,
\begin{equation}\label{eq:row-specialization}
 C_{k,s}(K;A,1)
 =\left(\frac{P_K(A)^s-1}{D_0}\right)^k\qquad(k\geq1).
\end{equation}
Both formulas are identities of Laurent polynomials in $A$.
\end{theorem}
\begin{proof}
At $q=1$, the symmetric quantum binomial becomes the ordinary binomial, and both $D_{r+i}$ and $D_{-s-j}$ become $D_0$. Thus
\[
 Z^+_{r,k}(A,1)=\binom{r}{k}D_0^k,\qquad
 Z^-_{s,l}(A,1)=\binom{s}{l}D_0^l.
\]
For an indeterminate $u$, expand $(1+u)^{rs}$ by choosing a subset of $[r]\times[s]$. Grouping subsets by their nonempty row and column sets yields
\[
 (1+u)^{rs}
 =1+\sum_{k=1}^r\sum_{l=1}^s
 \binom{r}{k}\binom{s}{l}
 \sum_{e=\max(k,l)}^{kl}n_{k,l,e}u^e.
\]
Take $u=D_0^2g_K(A)$ and compare this identity with \eqref{eq:double-expansion} and \eqref{eq:special-polynomial}. Triangular uniqueness over $\Q(A)$ gives \eqref{eq:graph-specialization}. Its exponents satisfy $2e-k-l\geq|k-l|\geq0$, proving the stated Laurent property directly. Formula \eqref{eq:graph-count} is the inclusion--exclusion formula excluding isolated vertices.

For fixed $s$, write $P_K(A)^{rs}=(1+(P_K(A)^s-1))^r$ and use the ordinary binomial theorem. Comparison with the specialized row expansion proves \eqref{eq:row-specialization}. The quotient in that formula is Laurent because $P_K^s-1$ is divisible by $P_K-1=D_0^2g_K$.
\end{proof}

The first examples are
\[
 \begin{aligned}
 E_{1,l}(A,1)&=D_0^{l-1}g_K^l,\\
 E_{2,2}(A,1)&=2g_K^2+4D_0^2g_K^3+D_0^4g_K^4.
 \end{aligned}
\]

\section{Examples and two-sided factors}\label{sec:consequences}
\subsection{Double-twist knots and further plats}
Formula~\eqref{eq:main} expresses the invariant through box factors, whereas \eqref{eq:K-det} uses a matrix of size $r+s$ in place of the $\binom{r+s}{s}$-dimensional rectangular channel space. Negative twists are computed by powers of its inverse.

The formula also extends to a zero twist after setting $\Tw_\lambda^{(0)}=\delta_{\lambda\vac}$. Directly from \eqref{eq:trace-closure},
\[
 \mathscr J_{0,n}=\mathscr J_{m,0}=1,
\]
because $P_0M_0^n=P_0$ and $M_1^mP_1=P_1$, using cyclicity in the first case as needed. These plats are unknots. On the shape side, $(B^0)_{\lambda\vac}=\delta_{\lambda\vac}$ gives the same answer.

\begin{example}\label{ex:fundamental}
For $r=s=1$,
\[
 b=\begin{pmatrix}1&0\\1&A^2\end{pmatrix},\qquad
 w_0=1,\qquad w_1=-A^2\{Aq\}\{A/q\}.
\]
For every integer $m$,
\[
 (b^m)_{10}=\frac{A^{2m}-1}{A^2-1}.
\]
Therefore
\begin{equation}\label{eq:fundamental}
 H_{[1]}(K_{m,n})
 =1-A^2\{Aq\}\{A/q\}
 \frac{(A^{2m}-1)(A^{2n}-1)}{(A^2-1)^2}.
\end{equation}
At $(m,n)=(1,1)$ this is $A^2(q^2+q^{-2})-A^4$; at $(-1,1)$ it is $1+\{Aq\}\{A/q\}$. These agree with the trefoil and figure-eight conventions of Section~\ref{sec:introduction}.
\end{example}

At $A=q^N$, $N\geq2s$, the same construction applies to a word $w$ in two invertible generators. Let $e(w)$ be its total exponent sum and let $w(M_0,M_1)$ denote the corresponding pure-braid operator. Closing it with the two vacuum pairings of Figure~\ref{fig:plat}, Propositions~\ref{prop:quantum-pair} and \ref{prop:RT-closure} give the reduced evaluation
\[
 \eta_s^{e(w)}
 \frac{\det\bigl((w(\theta,S\theta S)S)_{I_0,I_0}\bigr)}
 {\det(S_{I_0,I_0})}.
\]
Each generator is vacuum-normalized, so its framing correction is included in this formula.

\subsection{Torus knots and two-sided shape coefficients}\label{sec:obstruction}
\begin{example}\label{ex:torus}
For the left-handed torus knot $T(3,4)$, the Rosso--Jones formula \cite{RJ} gives
\[
 \begin{split}
 H_{1,1}={}&A^{10}-A^8(q^4+q^2+1+q^{-2}+q^{-4})\\
 &+A^6(q^6+q^2+1+q^{-2}+q^{-6}).
 \end{split}
\]
Consequently,
\[
 E_{1,1}=A^8-A^6(q^4+1+q^{-4})-A^4(q^2+q^{-2})-A^2.
\]
The higher coefficients are obtained by the exact recursion \eqref{eq:double-recursion} from the corresponding rectangular colored invariants.
\end{example}

The two-sided kernels \eqref{eq:two-sided-shape} occur in the
explicit double-twist formula. Their additional box factors are related
to the defect of the single-row differential expansion \cite{KM}.
The distinction between its general and fully factorized forms is
explained in \cite[Section~1, equations~(1.1)--(1.2)]{LM}.
Proposition~\ref{prop:obstruction} expresses the obstruction for $T(3,4)$
in terms of the double coefficient $E_{2,1}$.

The required Alexander specialization is the single-row case of Zhu's
torus-knot formula for hook partitions \cite[Theorem~1.2]{Zhu13}.

\begin{proposition}\label{prop:torus-Alexander}
For coprime positive integers $m,n$ and $r\geq1$, the symmetric-color torus-knot specialization is
\begin{equation}\label{eq:torus-Alexander}
 H_{(r)}(T(m,n);1,q)
 =\frac{\br{r}\br{mnr}}{\br{mr}\br{nr}}.
\end{equation}
The formula is the same for the mirror torus knot.
\end{proposition}
\begin{proof}
Using the quantum dimension \eqref{eq:quantum-dimension}, write
\[
 \psi_m(s_{(r)})=\sum_{\nu\vdash mr}c_{r,m}^\nu s_\nu,
 \qquad \psi_m(p_a)=p_{ma}.
\]
Here $p_a$ is the $a$th power-sum symmetric function.
In the left-handed convention, the Rosso--Jones formula \cite{RJ} is
\begin{equation}\label{eq:RJ-row}
 H_{(r)}(T(m,n);A,q)
 =\frac{A^{n(m-1)r}q^{mn r(r-1)}}{d_{(r)}(A,q)}
  \sum_{\nu\vdash mr}c_{r,m}^\nu
  q^{-2n\ct(\nu)/m}d_\nu(A,q).
\end{equation}
This is the zero-framed formula; $\ct(\nu)$ is the sum of its contents.

The order of vanishing of $d_\nu$ at $A=1$ is the number of diagonal boxes of $\nu$. Since $d_{(r)}$ has one such zero, only hooks
\[
 \nu_v=(mr-v,1^v),\qquad 0\leq v<mr,
\]
can contribute after division and specialization. The hook formula gives
\begin{equation}\label{eq:hook-dimension-limit}
 \left.\frac{d_{\nu_v}}{d_{(r)}}\right|_{A=1}
 =(-1)^v\frac{\br{r}}{\br{mr}},\qquad
 \ct(\nu_v)=\frac{mr(mr-1-2v)}2.
\end{equation}

The hook Adams coefficients are
\begin{equation}\label{eq:hook-adams}
 c_{r,m}^{\nu_v}=
 \begin{cases}(-1)^v,&0\leq v<m,\\0,&m\leq v<mr.\end{cases}
\end{equation}
To verify this, use the Frobenius character formula and the hook-character generating function \cite[Chapter~I, Section~3, Example~14, and Section~7]{Macdonald}. Since the character of $(r)$ is trivial, they give
\[
 \sum_{v=0}^{mr-1}c_{r,m}^{\nu_v}t^v
 =\frac1{1+t}\sum_{\mu\vdash r}\frac1{z_\mu}
       \prod_{a\in\mu}(1-(-t)^{ma}),
 \qquad z_\mu=\prod_{a\geq1}a^{m_a(\mu)}m_a(\mu)!.
\]
Here $m_a(\mu)$ is the multiplicity of the part $a$. The generating series in $r$ for the sum on the right before division by $1+t$ is
\[
 \exp\left(\sum_{a\geq1}\frac{u^a}{a}(1-(-t)^{ma})\right)
 =\frac{1-u(-t)^m}{1-u}.
\]
Its coefficient for every $r\geq1$ is $1-(-t)^m$. Division by $1+t$ gives $\sum_{v=0}^{m-1}(-t)^v$, proving \eqref{eq:hook-adams}.

Substituting \eqref{eq:hook-dimension-limit} and \eqref{eq:hook-adams} into \eqref{eq:RJ-row} leaves
\[
 \frac{\br{r}}{\br{mr}}q^{-nr(m-1)}
       \sum_{v=0}^{m-1}q^{2nrv}
 =\frac{\br{r}\br{mnr}}{\br{mr}\br{nr}}.
\]
Changing $n$ to $-n$ leaves the final expression unchanged and gives the mirror case.
\end{proof}

\begin{proposition}\label{prop:obstruction}
For $K=T(3,4)$, there is no family $F_\lambda(K;A,q)\in \R$ satisfying
\begin{equation}\label{eq:impossible-expansion}
 H_{r,s}=\sum_{\lambda\subseteq(r^s)}
 F_\lambda(K;A,q)\Psi_{\lambda;r,s}
 \qquad(r,s\geq1),
\end{equation}
with $F_\varnothing=1$. The obstruction already occurs at height one and widths one and two.
\end{proposition}
\begin{proof}
At height one, only row partitions occur, and
\[
 \Phi_{(1^k),1}(A,q^{-1})=\prod_{i=0}^{k-1}D_{i-1}.
\]
Comparing \eqref{eq:impossible-expansion} with the unique row expansion therefore gives
\[
 C_{k,1}=F_{(k)}\prod_{i=0}^{k-1}D_{i-1}.
\]
On the other hand, \eqref{eq:C-from-E} gives $C_{k,1}=D_{-1}E_{k,1}$ for $k>0$. In particular, \eqref{eq:impossible-expansion} would require
\begin{equation}\label{eq:necessary-D0}
 E_{2,1}=D_0F_{(2)}.
\end{equation}

Proposition~\ref{prop:torus-Alexander} gives
\[
 \begin{aligned}
 H_{1,1}(K;1,q)&=q^6-q^4+1-q^{-4}+q^{-6},\\
 H_{2,1}(K;1,q)&=q^{12}-q^8+1-q^{-8}+q^{-12}.
 \end{aligned}
\]
Substituting these two expressions into \eqref{eq:E11} and the width-two expansion yields
\[
 \begin{aligned}
 E_{1,1}(K;1,q)&=-(q^4+q^2+1+q^{-2}+q^{-4}),\\
 E_{2,1}(K;1,q)&=-(q^2-q^{-2})(q^4+q^2+1+q^{-2}+q^{-4})\neq0.
 \end{aligned}
\]
Since $D_0$ vanishes at $A=1$, this contradicts \eqref{eq:necessary-D0} for a Laurent polynomial $F_{(2)}$.
\end{proof}

\appendix
\section{Further identities for the Gaussian matrix}\label{app:height}
In this appendix $\QB ac$ denotes the Gaussian binomial with parameter $Q$. We work with finite particle cutoffs; triangularity makes the formulas compatible as the cutoff increases.

\begin{proposition}\label{prop:positive}
For increasing nonnegative sequences $a_1<\cdots<a_k$ and $c_1<\cdots<c_k$,
\[
 \det\left(\QB{a_i}{c_j}\right)_{i,j=1}^k\in\Z_{\geq0}[Q].
\]
\end{proposition}
\begin{proof}
Consider directed lattice paths with east edges of weight $Q^y$ at height $y$ and south edges of weight one. The generating function from $(c,c)$ to $(a,0)$ is $\QB ac$: separating the first step gives the recurrence
$F_{a,c}=F_{a-1,c-1}+Q^cF_{a-1,c}$ and the Gaussian boundary conditions. The sources on the diagonal and sinks on the horizontal boundary are compatibly ordered: any path family realizing a nonidentity matching contains an intersection. The Lindstr\"om--Gessel--Viennot cancellation \cite{Lindstrom,GV} therefore identifies the determinant with the sum over vertex-disjoint families with the ordered matching. Each surviving weight is a nonnegative power of $Q$.
\end{proof}

For $\lambda$ of length at most $s$, let $L_i(\lambda)=\lambda_i+s-i$ and define
\[
 \mathcal G_s^{(k)}(\lambda)=\prod_{i=1}^s\prod_{j=0}^{k-1}
 (1-q^{2(L_i(\lambda)+j+1)}).
\]
\begin{proposition}\label{prop:column}
For every $k\geq0$,
\begin{equation}\label{eq:column}
 B^{(s)}_{\lambda+k1^s,\mu+k1^s}(A,q)
 =A^{2sk}q^{2sk(k-s)}
 \frac{\mathcal G_s^{(k)}(\lambda)}{\mathcal G_s^{(k)}(\mu)}
 B^{(s)}_{\lambda\mu}(Aq^{2k},q).
\end{equation}
The identity is understood on any finite truncation containing its indices.
\end{proposition}
\begin{proof}
For $k=1$, every particle increases by one. Use the Gaussian ratio in the proof of Proposition~\ref{prop:height} in each entry of the minor. The column factors give $\mathcal G_s^{(1)}(\lambda)/\mathcal G_s^{(1)}(\mu)$. The content identity
\[
 \ct(\mu+1^s)-\ct(\mu)=|\mu|-\frac{s(s-1)}2
\]
then gives the scalar $A^{2s}q^{-2s(s-1)}$ and the parameter change $A\mapsto Aq^2$. Iterate. The successive scalars multiply to
$\prod_{j=0}^{k-1}(Aq^{2j})^{2s}q^{-2s(s-1)}
=A^{2sk}q^{2sk(k-s)}$, and the successive gauge factors are precisely $\mathcal G_s^{(k)}$.
\end{proof}

The height filtration can also be expressed by a block decomposition. Over a field $\K$, let
\[
 a=\begin{pmatrix}1&0\\v&D\end{pmatrix}
 \quad\text{on }\K e_0\oplus H_+,
\]
with $D$ invertible. Order the exterior space as
$(\bigwedge^{s-1}H_+\wedge e_0)\oplus\bigwedge^sH_+$. Then
\begin{equation}\label{eq:extension}
 C_s(a)=\begin{pmatrix}C_{s-1}(D)&0\\\Omega&C_s(D)\end{pmatrix},\qquad
 \Omega(\xi)=C_{s-1}(D)\xi\wedge v.
\end{equation}
If $\Omega_m$ denotes the lower-left block of $C_s(a)^m$, ordinary block multiplication proves, for all $m,n\in\Z$,
\begin{equation}\label{eq:cocycle}
 \Omega_{m+n}=\Omega_m C_{s-1}(D)^n+C_s(D)^m\Omega_n.
\end{equation}
The inverse case follows either by multiplying inverses or by setting $n=-m$. For $m>0$, the one-particle lower-left block is $\sum_{j=0}^{m-1}D^jv$, and the same wedge calculation gives $\Omega_m$ directly. The invariant exact-height subspace and its lower-height quotient in Proposition~\ref{prop:height} are instances of \eqref{eq:extension}.

At height two, for $a\geq b\geq1$ the extension from a row state $(c,0)$ to $(a,b)$ is
\[
 B^{(2)}_{(a,b),(c,0)}
 =A^{2c}q^{2c^2-c-a-b}
 \left(\QB{a+1}{c+1}-\QB b{c+1}\right),\qquad Q=q^2.
\]
For $0\leq c\leq a$, the difference has the positive expansion
\[
 \QB{a+1}{c+1}-\QB b{c+1}
 =\sum_{j=\max\{b,c\}}^a Q^{j-c}\QB jc.
\]
It follows by telescoping the Gaussian recurrence. For $0\leq c\leq a$, the lowest exponent is $\max\{b-c,0\}$, the degree is $(c+1)(a-c)$, and the leading coefficient is one. For $c>a$, the difference vanishes.

\section{The two-strand torus-knot formula}\label{app:rosso-jones}
The left-handed trefoil is the two-strand case of the Rosso--Jones torus-knot formula \cite{RJ}. We give its expression in our framing convention. Let $s_R$ be the Schur character and write its second Adams operation as
\begin{equation}\label{eq:adams}
 \psi_2(s_R)=\sum_{\nu\vdash2|R|}c_{R,2}^{\nu}s_\nu,
 \qquad \psi_2(p_j)=p_{2j}.
\end{equation}
Use the unreduced quantum dimension $d_\nu(A,q)$ from
\eqref{eq:quantum-dimension}.

\begin{proposition}\label{prop:trefoil-trace}
For every partition $R$, the reduced zero-framed left-handed trefoil satisfies
\begin{equation}\label{eq:trefoil-trace}
 H_R(K_{1,1};A,q)
 =\frac{A^{3|R|}q^{12\ct(R)}}{d_R(A,q)}
 \sum_{\nu\vdash2|R|}c_{R,2}^{\nu}q^{-3\ct(\nu)}d_\nu(A,q).
\end{equation}
\end{proposition}
\begin{proof}
At $A=q^N$, decompose $R\otimes R$ into irreducibles and multiplicity spaces. Ribbon balancing makes the square of the braiding on the $\nu$-multiplicity space the scalar $q^{c_\nu-2c_R}$. After removing its square root, the braiding is an involution deforming the classical flip. Its trace is therefore the difference between the symmetric and antisymmetric multiplicities. The character identity
$\Tr(\mathrm{flip}\circ(g\otimes g))=\Tr_R(g^2)$ identifies this difference with $c_{R,2}^{\nu}$.

The raw closure of the inverse cube of the braiding is consequently
\[
 \sum_\nu c_{R,2}^{\nu}q^{-\frac32(c_\nu-2c_R)}d_\nu.
\]
Its writhe is $-3$, so multiply by $q^{3c_R}$ to give zero framing and divide by $d_R$ to reduce. Using
$c_R=N|R|+2\ct(R)$ and $c_\nu=2N|R|+2\ct(\nu)$ gives \eqref{eq:trefoil-trace}. Stable-rank specialization yields the two-variable identity.
\end{proof}
The Frobenius character formula \cite[Chapter I, Section 7]{Macdonald} gives
\[
 c_{R,2}^{\nu}=\sum_{\mu\vdash|R|}\frac{\chi_R(\mu)\chi_\nu(2\mu)}{z_\mu},
 \qquad z_\mu=\prod_j j^{m_j(\mu)}m_j(\mu)!,
\]
where $\chi_\lambda(\mu)$ is the irreducible symmetric-group character indexed by $\lambda$ at cycle type $\mu$, $m_j(\mu)$ is the number of parts equal to $j$, and $2\mu$ is obtained by doubling every part of $\mu$.

\section*{Acknowledgements}
We would like to thank BICMR and the anonymous experts for their valuable
support and helpful contributions to this project. The first author was
supported by the National Natural Science Foundation of China (NSFC) under
Grant No.~12341105.

\end{document}